\documentclass[11pt,a4paper]{amsart}
\usepackage{amssymb,amsfonts,amsmath,mathrsfs}
\usepackage{color}
\usepackage{comment}

\newtheorem{prop}{Proposition}[section]
\newtheorem{theorem}[prop]{Theorem}
\newtheorem{cor}[prop]{Corollary}
\newtheorem{lemma}[prop]{Lemma}

\theoremstyle{definition}
\newtheorem{Def}[prop]{Definition}

\newtheorem{rem}[prop]{Remark}
\newtheorem{example}[prop]{Example}
\newtheorem{conjecture}[prop]{Conjecture}
\numberwithin{equation}{section}

\def\R{\Bbb R}

\def\Dx{\Delta_x}
\def\Nx{\nabla_x}
\def\Dt{\partial_t}
\def\({\left(}
\def\){\right)}
\def\eb{\varepsilon}

\def\Cal{\mathcal}
\def\l{\lambda}

\def\l2l2{\Cal{L}(L^2(\Omega))}
\def\dist{\operatorname{dist}}

\def\Bbb{\mathbb}
\def\<{\left<}
\def\>{\right>}
\begin{document}
\title[Wave equation with sign changing damping]{Deterministic and random attractors for a wave equation with sign changing damping}
\author[Chang, Li, Sun and Zelik] {Qingquan Chang${}^1$ , Dandan Li${}^1$ , Chunyou Sun${}^1$ , and Sergey Zelik${}^{1,2}$}

\begin{abstract} The paper gives a detailed study of long-time dynamics generated by weakly damped wave equations in bounded 3D domains where the damping exponent depends explicitly on time and may change sign. It is shown that in the case when the non-linearity is superlinear, the considered equation remains dissipative if the weighted mean value of the dissipation rate remains positive and that the conditions of this type are not sufficient in the linear case. Two principally different cases are considered. In the case when this mean is uniform (which corresponds to deterministic dissipation rates), it is shown that the considered system possesses smooth uniform attractors as well as non-autonomous exponential attractors. In the case where the mean is not uniform (which corresponds to the random dissipation rate, for instance, when this dissipation rate is generated by the Bernoulli process), the tempered random attractor is constructed. In contrast to the usual situation, this random attractor is expected to have infinite Hausdorff and fractal dimension. The simplified model example which demonstrates infinite-dimensionality of the random attractor is also presented.
\end{abstract}


\thanks{
 This work is partially supported by  the RSF grant   19-71-30004  as well as  the EPSRC grant EP/P024920/1 and NSFC grants No. 11471148, 11522109, 11871169.
  The authors  would also like to thank D. Turaev and V. Kalantarov for stimulating discussions.}

\address{${}^1$ \phantom{e}School of Mathematics and Statistics, Lanzhou University, Lanzhou  \\ 730000,
P.R. China}
\email{changqq12@lzu.edu.cn; lidd2008@lzu.edu.cn; sunchy@lzu.edu.cn}
\address{${}^2$ University of Surrey, Department of Mathematics, Guildford, GU2 7XH, United Kingdom.}
 \email{s.zelik@surrey.ac.uk}

\maketitle
\tableofcontents
\section{Introduction}
The paper gives a comprehensive study of the following semilinear wave equation
\begin{equation}\label{0.wave}
\Dt^2 u+\gamma(t)\Dt u-\Dx u+f(u)=g
\end{equation}
in a bounded smooth domain $\Omega$ of $\R^3$ endowed with Dirichlet boundary conditions. Here $\Dx$ is the Laplacian with respect to the variable $x\in\Omega$, $f(u)$ and $g(x)$ are given non-linear interaction function and the external force respectively, and $\gamma(t)$ is the dissipation rate which in contrast to the standard situation may change sign.
\par
Various types of equations in the form of \eqref{0.wave} is of a great  permanent interest. On the one hand, they model many important phenomena arising in a modern science, for instance, in quantum mechanics, see \cite{Shiff,Segal}; semiconductor devices (e.g., Josephson junctions, see \cite{Mazo} and references therein); propagation of waves in a transmission wire (the so-called telegraph equation, see \cite{Hayt1989,Marshall1990}); geophysical flows, see e.g.,
 \cite{A.Majda2003,J.Pedlosky1987}; mathematical biology, see e.g., \cite{K.P.Hadeler1996}, etc.
\par
On the other hand, this type of equations also possess a nice and deep mathematical theory branching initially very different topics such as inverse scattering, harmonic analysis, Strichartz type estimates, non-concentration estimates and Pohozhaev-Morawetz inequalities, etc., which makes the theoretical study of these equations also interesting and important, see \cite{Grill,SS,SS1,Sogge08} and references therein.
\par
It is believed that the analytic properties of solutions of \eqref{0.wave} depend mainly on the sign and the growth rate of the nonlinearity $f$ (the dissipative term $\gamma(t)\Dt u$ is subordinated and is not essential when the solutions on the finite time interval are considered). Namely,  if
\begin{equation}\label{0.p}
f(u)=u|u|^p+\text{"lower order terms"}
\end{equation}
which will be always assumed in this paper, these properties are related with the value of the exponent $p$ (the sign assumption is already incorporated in this condition and we will not consider the self-focusing case $f(u)\sim-u|u|^p$ in this work).
\par
The key tool for the mathematical study of these equations is the so-called energy identity which can be obtained by formal multiplication of \eqref{0.wave} by $\Dt u$ and integrating over $x$:
\begin{equation}\label{0.energy}
\frac d{dt}\(\frac12\|\Dt u\|^2_{L^2}+\frac12\|\Nx u\|^2_{L^2}+(F(u),1)-(g,u)\)+\gamma(t)\|\Dt u\|^2_{L^2}=0,
\end{equation}
where $F(u):=\int_0^uf(v)\,dv$, which determines the energy phase space of the problem considered
$$
E:=[H^1_0(\Omega)\cap L^{p+2}(\Omega)]\times L^2(\Omega),\ \ \xi_u:=\{u,\Dt u\}\in E
$$
and gives the natural control of the energy norm of a solution.
\par
In the energy subcritical and critical cases $p<2$ and $p=2$ respectively, the global existence and uniqueness of
the energy solutions (=solutions with finite energy)  can be obtained relatively easily since the energy  control \eqref{0.energy} is enough to treat the non-linear term $f(u)$ as a perturbation, see e.g. \cite{Lions, BV, Temam}. In contrast to this, in the supercritical case $p>4$, the global well-posedness of equation \eqref{0.wave} remains an open problem. Indeed, similarly to the 3D Navier-Stokes problem, we have here global existence of weak energy solutions (without uniqueness) and local existence of smooth solutions (which a priori may blow up in finite time), see \cite{CV02, ZelDCDS} and references therein for more details.
\par
The most interesting here is an intermediate case $2<p\le 4$. In this case, the only energy control \eqref{0.energy} is {\it not sufficient} to treat the non-linearity properly, but it remains subordinated to the linear part if more delicate space-time integrability properties for the solutions of the linear equation are used. The global well-posedness of energy solutions for the case of $\Omega=\R^n$ and $p<4$ has been obtained in \cite{Jorgens,Gin}.
\par
The quintic case $p=4$ is much more difficult since it is not clear how to "lift" the extra space-time integrability from linear equation to the non-linear one (at least in a straightforward way). For instance, in order to get uniqueness, one needs
the so-called Strichartz estimates for the solutions of \eqref{0.wave} like
\begin{equation}\label{0.str}
u\in L^4_{loc}(\R,L^{12}(\Omega)),
\end{equation}
see \cite{Strichartz,Sogge08,BSogge} and references therein. In order to prove that this Strichartz norm does not blow up in finite time, one usually exploits the so-called non-concentration estimates and Pohozhaev-Morawetz inequalities, see \cite{Sogge08,SS,SS1}. These results were first obtained for the case $\Omega=\R^3$, but then have been extended to the case of bounded domains as well based on relatively recent results for Strichartz estimates in bounded domains, see \cite{Plan1,Plan2,BSogge}. Note also that it is still an open problem whether or not the extra regularity \eqref{0.str} holds for any energy solutions, so we will refer to the energy solutions satisfying \eqref{0.str} as Shatah-Struwe (SS) solutions.
\par
We now discuss the asymptotic behavior of solutions of \eqref{0.wave} as $t\to\infty$. Of course, the structure of the dissipation term $\gamma(t)\Dt u$ plays a crucial role here. The most studied is the case where the dissipation rate $\gamma(t)$ is strictly positive and separated from zero, for instance, $\gamma\equiv const>0$. In this case, the energy identity \eqref{0.energy} gives in an immediate way the global Lyapunov functional. This functional trivializes the long-time dynamics and guarantees the convergence of trajectories to the set of equilibria. Nevertheless, even in this case proving the asymptotic compactness (which is necessary to get the convergence in the energy space) and smoothness of the so-called global attractor may be a nontrivial task. For energy subcritical case $p<2$ this result has been obtained in \cite{Haraux82,Haraux85} (see also \cite{BV,Temam,Hale1988,Hale04,Caraballo2010} and reference therein), the energy critical case $p=2$ has been treated in \cite{BV,Arrieta1992} (see also \cite{ZelDCDS} for treating the non-autonomous case and \cite{Li2017,Li2019} for degenerate case), the subcritical case $p<4$ has been studied in \cite{Feireisl} and \cite{Kapitanski} for the case of the whole space $\Omega=\R^3$ or periodic boundary conditions respectively. The critical case $p=4$ in general bounded domains has been considered in \cite{KSZ}. Note that the prove given there used in a crucial way the Lyapunov functional and cannot be extended to non-autonomous case. This drawback has been overcome in \cite{SZ} for the case of periodic boundary conditions using the so-called energy-to-Strichartz estimates, but obtaining such estimates in a critical case for general domains is still an open problem. This is the reason why we mainly consider the subcritical case $p<4$ in this paper. Mention also that some of the results can be extended to the supercritical case $p>4$ using the so-called trajectory attractors technique to overcome possible non-uniqueness, see \cite{CV02,ZelDCDS} for details.
\par
The next well-studied case is when the dissipation rate $\gamma$ is still non-negative but may be equal to zero at some nontrivial subset of $\Omega$. In this case, \eqref{0.energy} does not give immediately the global Lyapunov functional, so some new technique should come into play. In the case where $\gamma=\gamma(x)$ is degenerate, but non-negative damping, the results on the existence and further regularity of attractors are usually obtained based on a combination of two types of estimates: 1) Carleman type estimates which allows to get a global Lyapunov function and 2) the exponential decay estimates for the linear equation ($f(u)=0$) which requires the so-called geometric control conditions on the support of $\gamma$, see \cite{Bardos,Burg,Feireisl93,Hale04,Rauch,Zuazua} and references therein.  The complementary case when the degenerate dissipation rate $\gamma=\gamma(t)\ge0$ depends only on $t$ is also intensively studied, see \cite{Haraux05,Haraux2013,Martinez,Smith} and references therein although the general case $\gamma=\gamma(t,x)\ge0$ looks not properly studied yet.
\par

\par
In contrast to this, not much is known for the case where the dissipation rate may change sign. The key difference here is that the right-hand side of the energy equality \eqref{0.energy} is no more non-negative, so the global Lyapunov functional disappear and even proving the global boundedness or/and dissipativity of  solutions becomes a  non-trivial problem. It worth mentioning also that in absence of a Lyapunov functional the associated dynamics can easily be chaotic, this may be observed even in the simplest examples, see e.g., \cite{NO}.
\par
 To the best of our knowledge (at least for PDEs of the form \eqref{0.wave}), only the case where the negative part $\gamma_-(t):=-\max\{0,-\gamma(t)\}$ is small with respect to the positive part and can be treated as a perturbation are studied in the literature, see \cite{Fragnelli2008,Fragnelli2012,Freitas1996,Haraux05,Haraux2013,R.Joly2007,Martinez} and references therein. In addition, in all of the mentioned above papers, the linear equation (which corresponds to the case $f(u)=g=0$) is assumed to be stable, so the non-linearity is, in a sense, treated as a perturbation which does not affect the dissipation mechanism. Some exception from this is the paper \cite{KZstr} where doubly nonlinear strongly damped wave equation with sign-changing dissipation rate is considered.
 \par
 On the other hand, the classical model example here is  Van der Pol equation
 \begin{equation}\label{0.van}
 y''+(y^2-1)y'+y=0
 \end{equation}
 which describes generation of auto-oscillations in various physical systems arising, say, in radio-electronics, classical mechanics, biology, etc., see \cite{NO} and references therein (the PDE analogue of this equation as well as related Fitz-Hugh-Nagumo equation,  describes oscillatory processes in excitable media, see e.g., \cite{Car}). Equation \eqref{0.wave} can be considered as a simplified model for such problems where the nonlinear damping term  is replaced by $\gamma(t):=y^2(t)-1$ for some special solution  $y(t)$ of an ODE, e.g, of the Van der Pol equation or its multi-dimensional analogue.
\par
The example of the Van der Pol equation guesses that the assumption that the linear part of the equation must be stable is too restrictive and that the equation can be stabilized by the non-linear terms. As we will see later this is exactly the case for equation \eqref{0.wave}. Surprisingly, the global stability analysis of the non-linear problem \eqref{0.wave} with $f(u)$ satisfying \eqref{0.p} with $p>0$ is {\it simpler} than in the linear case and reasonable conditions for $\gamma(t)$ can be stated in this case.
\par
Thus, the ultimate goal of the present paper is to give a detailed study of equation \eqref{0.wave} in the super-linear case where $f$ satisfies \eqref{0.p} with $p>0$. Our main assumption on the dissipation rate is the following:
\begin{equation}\label{0.dissipative}
\liminf_{T\to\infty}\frac1T\int_{\tau-T}^\tau \(\frac12\gamma_+(t)-\frac{p+2}{p+4}\gamma_-(t)\)dt>0,\ \ \tau\in\R,
\end{equation}
where $\gamma_+=\max\{0,\gamma\}$, $\gamma_-=\gamma_+-\gamma$.
The heuristic arguments showing why this condition looks necessary are given in Section \ref{s1}.
\par
 We start with the discussion why the analogue of \eqref{0.dissipative} does not work in the linear case. The simplest model example here is the following ODE:
\begin{equation}\label{0.linODE}
y''+\gamma(t)y'+\omega^2 y=0
\end{equation}
with time-periodic $\gamma$, where the stability analysis  is already an interesting and non-trivial task. Indeed, the standard periodic change of variables, see Section \ref{s1}, transform this equation to the classical Mathieu-Hill's equation
\begin{equation}\label{0.MH}
y''+\<\gamma\>y'+ (\omega^2+\psi(t))y=0,
\end{equation}
where $\<\gamma\>$ is the mean value of $\gamma$ over the period and the periodic function $\psi(t)$ is calculated via $\gamma(t)$. This equation describes  resonances in
parametrically excited mechanical systems, see \cite{NO} and references therein; stability and bifurcations  of limit cycles, see e.g., \cite{KH}; etc. It also can be interpreted (at least when $\<\gamma\>$=0) as a Schr\"odinger equation with periodic potential which is central  in the quantum theory of solids, see e.g., \cite{Cycon}. It is well-known that the stability analysis for this equation is complicated and is not described by conditions like \eqref{0.dissipative}, see \cite{Magnus} and references therein. The case when $\gamma$ is not periodic (for instance, random), the situation becomes more difficult since the effects related with Anderson localization come into play, see \cite{Cycon}.
\par
In the present paper we demonstrate that assumptions like \eqref{0.dissipative} do not work in the linear case by proving the following result, see Proposition \ref{Proc1.ext} in Section \ref{s1}.

\begin{prop}\label{Prop0.lin} Let $a,b,\omega>0$ be arbitrary. Then there exists  $2\pi$-periodic function $\gamma\in L^1(0,2\pi)$ such that $\<\gamma_+\>=a$, $\<\gamma_-\>=b$ and equation \eqref{0.linODE} is exponentially unstable.
\end{prop}
However this instability mechanism does not work for the superlinear case $p>0$ since, in contrast to the linear case,  the frequency of internal oscillations grows with the energy growth and  parametric resonances
become {\it impossible} on higher energy levels. By this reason, when condition \eqref{0.dissipative} is satisfied, the dissipation becomes prevalent at higher energy levels which makes the equation globally dissipative, see Section \ref{s1} for details. To preserve this effect in the case where $\gamma(t)$ is not periodic, we need to assume an extra regularity assumption on $\gamma$ which guarantees that the frequency of "external oscillations" of $\gamma$ does not grow when time grows. Namely, we assume that $\gamma$ is translation-compact in the $L^1$-metric:
\begin{equation}\label{0.gamma}
\gamma\in L^1_{tr-c}(\R).
\end{equation}
Roughly speaking, assumption \eqref{0.gamma} means that $\gamma$ can be approximated in mean uniformly in time by smooth bounded functions,
see Section \ref{s2} for the rigorous definition. This condition is also natural since otherwise one can construct growing in time solutions for the non-linear case as well arguing exactly as in the proof of Proposition \ref{Prop0.lin}.
\par
As we show in the paper, there are two principally different case depending on whether or not the limit \eqref{0.dissipative} is uniform with respect to $\tau\in\R$. The first case (where it is uniform) is more standard and is  related with deterministic dissipation rate $\gamma(t)$, say, $\gamma(t)$ is periodic or quasi/almost periodic in time and the second one is natural for chaotic or random in time dissipation rate.
\par
We start with the uniform case. The key result in this case is the following uniform dissipative estimate for the SS solutions of equation \eqref{0.wave}.

\begin{theorem}\label{Th0.dis} Let the dissipation rate $\gamma(t)$ satisfies \eqref{0.gamma} and condition \eqref{0.dissipative} uniformly with respect $\tau\in\R$. Assume also that $g\in L^2(\Omega)$ and the nonlinearity $f$ satisfies \eqref{0.p} for some $0<p\le4$. Then, for every $\xi_\tau\in E$ there exist a unique SS solution $u(t)$, $t\ge\tau$, of problem \eqref{0.wave} with the initial data $\xi_u\big|_{t=\tau}=\xi_\tau$ and the following estimate holds:
\begin{equation}\label{0.udis}
\|\xi_u(t)\|_E^2\le Q(\|\xi_u(\tau)\|_E^2)e^{-\alpha (t-\tau)}+Q(\|g\|_{L^2}),\ \ t\ge\tau\in\R,
\end{equation}
where the positive constant $\alpha$ and monotone function $Q$ are independent of $t$, $\tau$, $u$ and $g$. Here and below $\xi_u(t):=\{u(t),\Dt u(t)\}$.
\end{theorem}
The proof of this result is given in Section \ref{s3}.
\par
The dissipative estimate \eqref{0.udis} allows us to apply the main techniques of the attractors theory to equation \eqref{0.wave} in a more or less standard way. Indeed, Theorem \ref{Th0.dis} allows us to define the dynamical process $U(t,\tau)$, $t\ge\tau$, in the energy space $E=H^1_0(\Omega)\times L^2(\Omega)$ (under the condition $p\le4$, we have Sobolev's embedding $H^1\subset L^{p+2}$, so the term $L^{p+2}$ is not necessary in the definition of $E$) and study its attractors. Since the equation considered depends explicitly on time, we need to use the proper extensions of a global attractor to the non-autonomous case. One of possible extensions is the so-called {\it uniform} attractor, see \cite{CV02} and references therein. By definition a uniform attractor $\mathcal A_{un}$ is a minimal compact set in $E$ which attracts all bounded subsets in $E$ uniformly with respect to $\tau\in\R$. Namely, for every bounded set $B$,
\begin{equation}\label{0.unat}
\lim_{s\to\infty}\sup_{\tau\in\R}\dist_E(U(\tau+s,\tau)B,\mathcal A_{un})=0,
\end{equation}
where $\dist$ stands for the Hausdorff distance in $E$, see Section \ref{s3} for more details. Then, the following theorem is proved in Section \ref{s3}.

\begin{theorem}\label{Th0.uatr} Let the assumptions of Theorem \ref{Th0.dis} and let, in addition, $p<4$. Then, the dynamical process $U(t,\tau)$ generated by solution operators of problem \eqref{0.wave} possesses a uniform attractor $\mathcal A_{un}$ which is a bounded set in the higher energy space $E^1:=[H^2(\Omega)\cap H^1_0(\Omega)]\times H^1_0(\Omega)$.
\end{theorem}
To describe the structure of a uniform attractor, we need as usual (see \cite{CV02} for details) to consider not only equation \eqref{0.wave}, but also all time shifts of it together with their limits in the proper topology. Namely, we need to consider the hull $\mathcal H(\gamma)$ of the initial dissipation rate $\gamma$:
\begin{equation}
\mathcal H(\gamma):=[T_h\gamma,\ h\in\R]_{L^1_{loc}(\R)},\ \ (T_h\gamma)(t):=\gamma(t+h),
\end{equation}
where $[\cdot]_V$ stands for the closure in $V$. In particular, assumption \eqref{0.gamma} implies that $\mathcal H(\gamma)$ is compact in $L^1_{loc}(\R)$. For every $\eta\in \mathcal H(\gamma)$, we consider equation \eqref{0.wave} with $\gamma$ replaced by $\eta$ and denote by $\mathcal K_\eta\subset L^\infty(\R,E)$ the set of all solutions of this equation defined for all $t\in\R$ and bounded as $t\to-\infty$, the so-called kernel of this equation in the terminology of Chepyzhov and Vishik, see \cite{CV02}. Then, the uniform attractor $\mathcal A_{un}$ of problem \eqref{0.wave} can be described as follows:
\begin{equation}
\mathcal A_{un}=\cup_{\eta\in\mathcal H(\gamma)}\mathcal K_\eta\big|_{t=0}.
\end{equation}
Moreover, following the general procedure, we may define the kernel sections
\begin{equation}\label{0.ks}
\mathcal K_\eta(\tau):=\mathcal K_\eta\big|_{t=\tau},\ \ \eta\in\mathcal H(\gamma),\ \ \tau\in\R.
\end{equation}
Then, as known (see \cite{CV94,CV02} for details),  these sections are compact in $E$ and possess the strict invariance property:
$$
U_\eta(t,\tau)\mathcal K_\eta(\tau)=\mathcal K_\eta(t),
$$
where $U_\eta(t,\tau)$ is the dynamical process generated by equation \eqref{0.wave} with $\gamma$ replaced by $\eta\in\mathcal H(\gamma)$, and enjoy the so-called pullback attraction property:
\begin{equation}\label{0.pull}
\lim_{s\to\infty}\dist_E(U_\eta(\tau,\tau-s)B,\mathcal K_\eta(\tau))=0.
\end{equation}
By this reason, the introduced family of kernel sections $\mathcal K_\eta(t)$, $t\in\R$, are often referred as a pullback attractor associated with the dynamical process $U_\eta(t,\tau)$, see \cite{CV94,CV02, Kloeden, Carvalho2012} for more details.
\par
Like global attractors for autonomous case, these kernel sections are usually compact and have finite fractal and Hausdorff dimension, but in contrast to uniform attractors, the rate of attraction in \eqref{0.pull} is typically {\it not uniform} with respect to $\tau\in\R$ (and $\eta\in\mathcal H(\gamma)$). By this reason, the forward in time attraction fails in general. Moreover, as elementary examples show, an exponentially repelling equilibrium may easily be a pullback "attractor" for  a dynamical process considered.
\par
One of the ways to overcome this drawback is to use the concept of an exponential attractor introduced in \cite{EFNT} and extended to the non-autonomous case in \cite{Mir,EMZ} (see also the survey \cite{MirZel} and references therein). By definition, a non-autonomous exponential attractor $\mathcal M_\eta(t)$, $t\in\R$, for the dynamical process $U_\eta(t,\tau)$ is a {\it semi-invariant} family of compact sets which have finite Hausdorff and fractal dimensions and possesses a {\it uniform} exponential attraction property, namely, there exist a positive constant $\alpha$ and monotone function $Q$ such that, for every bounded set $B$ of $E$,
\begin{equation}\label{0.exp}
\dist_E(U_\eta(\tau+s,\tau)B,\mathcal M_\eta(\tau+s))\le Q(\|B\|_E)e^{-\alpha s}
\end{equation}
uniformly with respect to $\tau\in\R$ (and actually also with respect to $\eta\in\mathcal H(\gamma)$).
In particular, as not difficult to see $\mathcal K_\eta(t)\subset\mathcal M_\eta(t)$ if the exponential attractor exists. We also emphasize that, in contrast to the kernel sections, the non-autonomous exponential attractor $\mathcal M_\eta(t)$, $t\in\R$, is {\it not only} pullback attracting, but also {\it forward in time} (exponentially) attracting.
\par
The next theorem, proved in section \ref{s4} establishes the existence of a non-autonomous exponential attractor for the wave equation \eqref{0.wave}.

\begin{theorem}\label{Th0.exp} Let the assumptions of Theorem \ref{Th0.uatr} holds and let, in addition, the dissipation rate $\gamma$ is more regular: $\gamma\in L^{1+\eb}_b(\R)$ for some $\eb>0$, see Section \ref{s4} for details. Then, the dynamical processes $U_\eta(t,\tau)$, $\eta\in\mathcal H(\gamma)$ associated with wave equation \eqref{0.wave} possess non-autonomous exponential attractors $\mathcal M_\eta(t)$ which are bounded sets of the higher energy space $E^1$.
\end{theorem}

We now turn to the second (probably more interesting) case where the dissipativity assumption \eqref{0.dissipative} is not uniform with respect to $\tau\in\R$. In this case, typically, the dissipation is not strong enough to provide boundedness of trajectories and dissipativity forward in time, so the uniform attractor cannot exist. Moreover, the kernels $\mathcal K_\eta(t)$ defined as above via all {\it bounded} solutions may be either empty at all or too small to get any type of attraction, so the theory should be properly modified.
\par
 A natural way to overcome this problem which comes from the theory of random attractors (see \cite{CF,Kloeden, Carvalho2012} and references therein) is to replace bounded trajectories by tempered ones and respectively bounded sets by tempered sets. Namely, a complete trajectory $u(t)$, $t\in\R$, of problem \eqref{0.wave} is tempered if $\|\xi_u(t)\|_E$ grows as $t\to-\infty$ slower than any exponent and a family of bounded sets $B(t)$, $t\in\R$, is tempered if $\|B(t)\|_E$ grows as $t\to-\infty$ slower than any exponent. Then, the theory of kernel sections developed in \cite{CV94,CV02} can be naturally extended to the tempered case by considering {\it tempered} kernels (=sets of all tempered complete trajectories) and tempered kernel sections (= tempered pullback attractors), see \cite{Kloeden, Carvalho2012} and references therein, and this is exactly the key technical tool which we need to treat wave equation \eqref{0.wave} in the non-uniform case, see Section \ref{s5} for more details.
\par
However, as in a bounded case, tempered kernel sections have an intrinsic drawback related with the absence of attraction forward in time which  disappears in random case where forward attraction usually holds in probability. Keeping also in mind that the non-uniformity with respect to $\tau\in\R$ in the dissipative condition \eqref{0.dissipative} naturally appears when the dissipation rate is random (or chaotic), we introduce the required random formalism from the very beginning. Namely, we assume that there is a Borel probability measure $\mu$  on the hull $\mathcal H(\gamma)$ such that it is invariant and {\it ergodic} with respect to time shifts
\begin{equation}
T_h:\mathcal H(\gamma)\to\mathcal H(\gamma),\ \ h\in\R,\ \ (T_h\eta)(t)=\eta(t+h).
\end{equation}
Then assumption \eqref{0.dissipative} will be replaced by
\begin{equation}\label{0.mdis}
\int_{\eta\in\mathcal H(\gamma)}\(\int_0^1\frac12\eta_+(t)-\frac{p+2}{p+4}\eta_-(t)\,dt\)\mu(d\eta)>0
\end{equation}
and the initial assumption \eqref{0.dissipative} will hold for {\it almost all}
 $\eta\in\mathcal H(\gamma)$ by the Birkhoff ergodic theorem.
\par
Recall that a $\mu$-measured set valued function $\eta\to\mathcal A(\eta)\subset E$ is called tempered random attractor for the family $U_\eta(t,\tau):E\to E$ of dynamical processes if
\par
1) $\mathcal A(\eta)$ are well-defined and compact in $E$ for almost all $\eta\in\mathcal H(\gamma)$;
\par
2) The family of bounded sets $t\to \mathcal A(T_t\eta)$ is tempered for almost all $\eta\in\mathcal H$
\par
3) It is strictly invariant: $U_\eta(t,0)\mathcal A(\eta)=\mathcal A(T_t\eta)$, $t\ge0$;
\par
4) For any other measured tempered random set $\eta\to B(\eta)$, we have
$$
\lim_{s\to\infty}\dist_E(U_\eta(0,-s)B(T_{-s}\eta),\mathcal A(\eta))=0
$$
for almost all $\eta\in\mathcal H(\gamma)$.
\par
The next theorem proved in Section \ref{s5} gives the existence of a tempered random attractor for equation \eqref{0.wave}.

\begin{theorem}\label{Th0.ran} Let $g\in L^2(\Omega)$, the non-linearity $f$ satisfy \eqref{0.p} with $0<p<4$ and $\gamma$ satisfy \eqref{0.gamma}. Assume also that the Borel probability measure $\mu$ on $\mathcal H(\gamma)$ is invariant and ergodic with respect to time shifts and assumption \eqref{0.mdis} is satisfied. Then the family of dynamical processes $U_\eta(t,\tau)$, $\eta\in\mathcal H(\gamma)$ possesses a tempered random attractor $\mathcal A(\eta)$. Moreover, this random attractor is attracting forward in time in sense of convergence in measure:
\begin{equation}\label{0.muf}
\mu-\lim_{t\to\infty}\dist_E(U_\eta(t,0)B(\eta),\mathcal A(T_t\eta))=0,
\end{equation}
for every tempered random set $B(\eta)$, see Section \ref{s5}.
\end{theorem}
As usual, we get the random attractor by constructing the tempered kernel sections $\mathcal K_\eta(t)$ for almost all $\eta\in\mathcal H(\gamma)$ and then set $\mathcal A(\eta):=\mathcal K_\eta(0)$.
\par
Our key model example of random dissipation rate  is the following piece wise constant function:
\begin{equation}\label{0.ber}
\eta(t):=\eta_n,\ \ t\in[n,n+1),\ \ n\in\Bbb Z,
\end{equation}
where $\{\eta_n\}_{n\in\Bbb Z}\in \Gamma:=\{a,-b\}^{\Bbb Z}$ is a Bernoulli scheme with two symbols $a>0$ and $b>0$. We assume that the value $a$ has probability $q$ to appear (for some $0<q<1$) and the remaining value $-b$ appears with probability $1-q$ and let $\mu$ be a product measure on the Bernoulli scheme $\Gamma$. Then, as known, see e.g., \cite{KH}, this measure is invariant and ergodic with respect to discrete shifts $T_l:\Gamma\to\Gamma$, $l\in\Bbb Z$. Moreover, $\Gamma$ endowed by the Tichonoff topology is compact and possesses a dense trajectory which we take as the initial $\gamma$ and construct $\gamma(t)$ by \eqref{0.ber}. Then the hull $\mathcal H(\gamma)$ will generate the whole Bernoulli scheme $\Gamma$. Note also that discrete shifts
on $\Gamma$ are conjugated to the discrete shifts on the hull $\mathcal H(\gamma)$. Thus, the conditions of Theorem \ref{Th0.ran} will be satisfied if we verify \eqref{0.mdis}. The straightforward calculations show that it is satisfied iff
\begin{equation}\label{0.disB}
aq-\frac{2(p+2)}{p+4}b(1-q)>0.
\end{equation}
Thus, under this assumption, equation \eqref{0.wave} with the dissipation rate generated by the Bernoulli process possesses a tempered random attractor.
\par
Up to the moment, the application of the random attractors theory to the  case of equation \eqref{0.wave}  is more or less standard. However, there is a principal difference here. Namely, in contrast to the usual situation, we cannot guarantee that the constructed random attractor has a finite first moment, moreover, we expect that
\begin{equation}\label{0.inf}
\int_{\eta\in\mathcal H(\gamma)}\|\mathcal A(\eta)\|_{E}\mu(d\eta)=\infty.
\end{equation}
At least we have this equality for the random absorbing ball constructed in the proof of Theorem \ref{Th0.ran} in the case of Bernoulli process which satisfies \eqref{0.disB} and does not satisfy the stronger assumption
\begin{equation}\label{0.good}
\ln(e^{-a}q+e^{\frac{2(p+2)}{p+4}b}(1-q))<0
\end{equation}
and we do not see any reasons why this can be improved. This has a drastic impact on the dynamics of the considered random system. Indeed, if \eqref{0.inf} is infinite, there are no reasons to expect that the random Lyapunov exponents (see \cite{Arnold,CF1,Deb}) will be finite and this, in turn, may lead to the infinite dimensionality of the corresponding random attractor $\mathcal A(\eta)$.
\par
Our conjecture is that in this case the random attractor $\mathcal A(\eta)$ is indeed infinite-dimensional for almost all $\eta\in\mathcal H(\gamma)$. Since the fact that random/stochastic perturbation of a "good" dissipative system with finite-dimensional attractor may lead to infinite-dimensional dynamics potentially may have a fundamental impact on the theory of random dynamical systems and, to the best of our knowledge, has been not considered before, we give in Section \ref{s6} a simple model example demonstrating this effect. Namely, we consider the following infinite system of ODEs in a Hilbert space $H=l_2$:
\begin{equation}
u_1'+\eta(t)u_1=1,\ \ u_n'+n^4u_n=u_1u_n-u_n^3,\ \ n=2,\cdots.
\end{equation}
In this case, if $\eta$ is generated by the Bernoulli process and such
that \eqref{0.disB} is satisfied and \eqref{0.good} is not satisfied (both for $p=0$), the associated random attractor exists, but has infinite Hausdorff and fractal dimension, see Section \ref{s6} for details. Note that in this case the attractor is clearly finite dimensional in the deterministic case, e.g., if $\gamma=const>0$ or satisfies condition \eqref{0.dissipative} uniformly with respect to $\tau\in\R$.
\par
The paper is organized as follows. Section \ref{s1} mainly consists of heuristic arguments demonstrating that the posed assumptions are natural and reasonable. In addition, Proposition \ref{Prop0.lin} is proved and the case where \eqref{0.wave} is an ODE is rigorously treated there.
\par
Section \ref{s2} is devoted to the proof of the key dissipative estimate \eqref{0.udis}. The rigorous definitions for weak energy and Shatah-Struwe solutions and more rigorous discussion of known result about global solvability of \eqref{0.wave} is also presented there.
\par
The asymptotic compactness estimates which guarantees that a bounded ball of $E^1$ attracts exponentially all solutions of \eqref{0.wave} in the uniformly dissipative case is presented in Section \ref{s3} and uniform and exponential attractors for this case are constructed in Section \ref{s4}.
\par
The non-uniform dissipation and random cases are considered in Section \ref{s5}. In particular, Theorem \ref{Th0.ran} is proved there. Finally, the related model example where the dimension of a random attractor is infinite is studied in Section \ref{s6}.

\section{Preliminaries and heuristics}\label{s1}
In this section, we show that the conditions on the mean value of the dissipation rate $\gamma(t)$ are not relevant for the case of linear equations and give some evidence that they are natural and, in a sense, necessary in the super-linear case.

\subsection{Linear ODE} We start with the simplest, but already very non-trivial case of a scalar equation:
\begin{equation}\label{1.scalar}
y''(t)+\gamma(t) y'(t)+ y(t)=0,
\end{equation}
where we assume for simplicity that $\gamma(t)$ is smooth and $T$-periodic. Let $\gamma_0:=\frac1T\int_0^T\gamma(t)\,dt$ be the mean value of $\gamma$. Then the standard time-periodic change of variables
$$
y(t)=e^{-\frac12\int_0^t(\gamma(s)-\gamma_0)\,ds}z(t)
$$
reduces the equation to the damped version  of the classical Mathieu-Hill's equation
\begin{equation}\label{1.mh}
z''+\gamma_0z'+(1+\psi(t))z=0,\ \ \ \psi(t):=\frac14\(\gamma_0^2-\gamma^2(t)-2\gamma'(t)\).
\end{equation}
The most studied is the non-dissipative case $\gamma_0=0$ and $\psi(t)=\eb\sin(\omega t)$ which corresponds to the original  Mathieu's equation. Then the instability in this equation is caused by the so-called parametric resonances and the instability zone (where the exponentially growing/decaying solutions of \eqref{1.mh}) on the $(\omega,\eb)$-plane touches the $\eb=0$ line in infinitely many points $\omega=n/2$, $n\in\Bbb Z$ and forms the famous Arnold's tongues, see \cite{KH,Magnus,NO} for more details. For non-zero dissipation rate $\gamma_0>0$ the number of tongues touching $\eb=0$ becomes finite, but it grows as $\gamma_0\to0$.
\par
The above described picture remain similar for a general {\it periodic} function $\psi$, but becomes much more complicated if the periodicity assumption is broken, see \cite{Cycon}.
\par
Thus, the stability of equation \eqref{1.scalar} is an interesting and delicate problem and it is unlikely that more or less sharp conditions for it can be formulated in a simple way. The next proposition  gives an alternative way to generate instability directly in equation \eqref{1.scalar} and has an independent interest.
\par
Consider the class of functions
$$
\Gamma_{a,b}:=\left\{\gamma\in L^1_{per}(0,2\pi),\ \int_0^{2\pi}\gamma_+(t)\,dt=a,\ \int_0^{2\pi}\gamma_-(t)\,dt=b\right\}
$$
where $a,b\ge0$ be two given numbers, $\gamma_+=\max\{\gamma,0\}$ and $\gamma_-=\gamma_+-\gamma$. Then the following result holds.
\begin{prop}\label{Proc1.ext} Let $\gamma\in \Gamma_{a,b}$ and let $\mu_+(\gamma)$ and $\mu_-(\gamma)$ be the maximal and minimal Lyapunov exponents for equation \eqref{1.scalar} respectively. Then
\begin{equation}\label{1.extreme}
\sup_{\gamma\in\Gamma_{a,b}}\mu_+(\gamma)=\frac b{2\pi},\ \ \inf_{\gamma\in\Gamma_{a,b}}\mu_-(\gamma)=-\frac a{2\pi}.
\end{equation}
\end{prop}
\begin{proof} We first mention that due to the Liouville theorem,
$$
\mu_+(\gamma)+\mu_-(\gamma)=\frac{b-a}{2\pi},
$$
so we only need to prove the first equality of \eqref{1.extreme}. Let us start with the estimate from above. To this end, multiplying equation \eqref{1.scalar} by $y'(t)$ we get
\begin{multline}\label{1.en1}
\frac12\frac d{dt}\(y'(t)^2+y(t)^2)\)=\\=-\gamma(t)y'(t)^2\le \gamma_-(t)y'(t)^2\le\gamma_-(t)\(y'(t)^2+y(t)^2\).
\end{multline}
Integrating this estimate, we arrive at
$$
y'(2\pi)^2+y(2\pi)^2\le e^{2b}\(y'(0)^2+y(0)^2\)
$$
which implies that $\mu_+(\gamma)\le \frac b{2\pi}$.
\par
For the lower bound, we use an explicit construction of the function $\gamma_h(t)\in\Gamma_{a,b}$ depending on a small parameter $h$. Namely, let
\begin{equation}
\gamma_h(t)=\begin{cases}a/h,\ \ t\in[0,h],\\0,\ \ t\in(h,\pi)\cup(\pi+h,2\pi),\\
-b/h,\ \ t\in[\pi,\pi+h].\end{cases}
\end{equation}
To find the Lyapunov exponents, we need to compute the eigenvalues of the period map related with this choice of the function $\gamma_h$. Let us denote by $U_h(t,s)$ the solution matrix related with equation \eqref{1.scalar}, i.e.,
$$
\(\begin{matrix} y(t)\\y'(t)\end{matrix}\):=U_h(t,s)\(\begin{matrix} y(s)\\y'(s)\end{matrix}\),
$$
decompose the desired period map as follows:
$$
P(h)=U_h(2\pi,\pi+h)U_h(\pi+h,\pi)U_h(\pi,h)U_h(h,0)
$$
and find the limit $P(h)$ as $h\to0$. Obviously,
$$
U_h(\pi,h)=U_h(2\pi,\pi+h,)=\(\begin{matrix}0&-1\\-1&0\end{matrix}\)+O(h)
$$
and the matrices $U_h(h,0)$ and $U_h(\pi,\pi+h)$ also coincide up to changing $a$ to $-b$. Finally, the straightforward computations involving the explicit formula for the solution give
$$
U_h(h,0)=\(\begin{matrix} 1&0\\0&-e^{-a}\end{matrix}\)+O(h)
$$
and therefore
\begin{multline*}
P(h)=\(\begin{matrix}0&-1\\-1&0\end{matrix}\)\(\begin{matrix} 1&0\\0&-e^{b}\end{matrix}\) \(\begin{matrix}0&-1\\-1&0\end{matrix}\)\(\begin{matrix} 1&0\\0&-e^{-a}\end{matrix}\) +O(h)=\\=\(\begin{matrix}e^b&0\\0&e^{-a}\end{matrix}\)+O(h).
\end{multline*}
Thus, $\mu_+(\gamma_h)=\frac b{2\pi}+O(h)$ and the proposition is proved.
\end{proof}
\begin{rem}\label{Rem1.not} Using the energy arguments as in the proof of the upper bound together with the fact that the solution $y(t)$ cannot be zero identically on any interval, we see that the supremum and infimum in \eqref{1.extreme} are not attained if $ab\ne0$. We also note that the proved result shows that in any class $\Gamma_{a,b}$, $ab\ne0$ there is an element $\gamma$ with positive Lyapunov exponent.
\end{rem}
\subsection{Super-linear ODE} As we have seen, the conditions on the mean value of $\gamma_+$ or $\gamma_-$ are not sufficient for establishing the absence of growing solutions. Surprisingly, the situation is essentially simpler in the case of non-linear equations with super-linear non-linearities. As we have already mentioned in the introduction, the reason for this is that, in contrast to the linear case, the frequency of internal oscillations grows when the energy grows, so adding the energy to the system destroys the conditions for  parametric resonances. We start with a heuristic derivation for the dissipative estimate which will be rigorously justified later.
\par
Let us consider the equation
\begin{equation}\label{1.non-lin}
y''(t)+\gamma(t)y'(t)+y(t)|y(t)|^p=0.
\end{equation}
When $\gamma=0$, the equation is invariant with respect to scaling $t\to tE^{\frac{-p}{2(p+2)}}$, $u\to uE^{\frac1{p+2}}$, so the energy $E$ and the frequency $\omega$ of internal oscillations are related as
$$
\omega\sim E^{\frac{p}{2(p+2)}}
$$
and the dissipative term $\gamma y'$ will be of order $E^{\frac{-p}{2(p+2)}}$ in the scaled time, so it cannot change the oscillatory nature of the solutions (at least if $\gamma$ is bounded). By this reason, all solutions of equation \eqref{1.non-lin} will oscillate rapidly in time on high energy levels.
\par
We now write down the energy equality
\begin{equation}\label{1.ener}
\frac d{dt} E(t):=\frac d{dt}\(\frac12 y'(t)^2+\frac1{p+2}|y(t)|^{p+2}\)=-\gamma(t) y'(t)^2.
\end{equation}
From this equality we see that the total energy $E(t)$ is not oscillatory and, moreover, if we fix a small enough interval $t\in[0,\eb]$, we get
\begin{equation}\label{1.app1}
E(t)\approx E(0),\ \text{i.e.,}\ \ |E(t)-E(0)|\le C\eb,\ \ t\in[0,\eb].
\end{equation}
In contrast to this, kinetic energy $E_k(t):=\frac12|y'(t)|^2$ as well as the potential one $E_{p}(t):=E(t)-E_k(t)$ is oscillatory, so the right-hand size of \eqref{1.ener} can be averaged (if the initial energy $E(0)$ is large enough and $\eb$ is fixed) and we get
\begin{equation}\label{1.av}
E(\eb)-E(0)=-\int_0^\eb\gamma(t)y'(t)^2\,dt\approx -2\int_0^\eb\gamma(t)\,dt\<E_k\>,
\end{equation}
 where $\<E_k\>:=\frac1\eb\int_0^\eb E_k(t)\,dt$ is the mean of the kinetic energy on the interval $t\in[0,\eb]$. To find this average, we multiply equation \eqref{1.non-lin} by $y(t)$ and integrate in time to get
 \begin{equation}\label{1.eb}
 (p+2)\<E_p\>-2\<E_k\>=-\<\gamma y'y\>+\frac{y'(0)y(0)-y'(\eb)y(\eb)}\eb:=H.
 \end{equation}
 Using now the fact that the potential energy is super-linear, with the help of \eqref{1.app1} and  the Young inequality, we get
 \begin{equation}\label{1.hold}
 |H|\le \beta E(0)+C_\beta,
 \end{equation}
 where $\beta>0$ is arbitrary and $C_\beta$ is independent of $E(0)$. Thus, if the initial energy $E(0)$ is large enough, we may write
$$
 (p+2)\<E_p\>\approx 2\<E_k\>+C,\ \ C=C_{\beta,\eb}
$$
which together with the energy balance $\<E_k\>+\<E_p\>\approx E(0)$ gives the fundamental relation
\begin{equation}\label{1.bal1}
\<E_k\>\approx\frac{p+2}{p+4}\<E\>+C \approx \frac{p+2}{p+4}E(0)+C.
\end{equation}
Inserting this relation to  energy identity \eqref{1.av}, we finally get
$$
E(\eb)\approx \(1-\frac{2(p+2)}{p+4}\int_0^\eb \gamma(t)\,dt\)E(0)+C.
$$
Repeating these arguments on the time interval $t\in[n\eb,(n+1)\eb]$, we finally arrive at
\begin{equation}\label{1.discrete}
E((n+1)\eb)\approx \(1-\frac{2(p+2)}{p+4}\int_{n\eb}^{(n+1)\eb} \gamma(t)\,dt\)E(n\eb)+C.
\end{equation}
Finally, if we assume that
\begin{equation}\label{1.normal}
\lim_{\eb\to0}\sup_{n\in\Bbb N}\int_{n\eb}^{(n+1)\eb}\gamma(t)\,dt=0
\end{equation}
and
\begin{equation}\label{1.positive}
\liminf_{T\to\infty}\inf_{t\ge0}\frac1T\int_t^{t+T}\gamma(s)\,ds>0.
\end{equation}
we may fix $\eb>0$ small enough and use that $\ln(1+x)\approx x$ to infer that
$$
E(n\eb)\le CE(0)e^{-\alpha n}+C_*
$$
for some positive constants $\alpha$ and $T$. This gives the desired dissipation. Analogously, if
 \begin{equation}\label{1.negative}
\limsup_{T\to\infty}\sup_{t\ge0}\frac1T\int_t^{t+T}\gamma(s)\,ds<0,
\end{equation}
the solutions of \eqref{1.non-lin} will grow exponentially at least if the initial energy is large enough. Thus, the mean value of the dissipation coefficient determines indeed whether or not the corresponding equation is dissipative.
\begin{rem}\label{Rem1.proof} Assumption \eqref{1.normal} is crucial for dissipativity. Indeed, it guarantees that the dissipation rate oscillates not too fast and makes possible the averaging with respect to the internal oscillations (in the sequel, we replace it by a bit stronger assumption that $\gamma$ is {\it translation-compact} in $L^1_b(\R)$). It is not difficult to see that if this condition is violated, we can destabilize equation \eqref{1.non-lin} similarly to the linear case (see the proof of Proposition \ref{Proc1.ext}), but using the kicks with smaller and smaller $h$.
  \end{rem}
We now give a rigorous proof for dissipativity of equation \eqref{1.non-lin} under a bit stronger (than \eqref{1.normal}) assumption that $\gamma$ has a bounded derivative which be relaxed later.
\begin{prop}\label{Prop1.dis} Let the function $\gamma(t)$ satisfy assumption \eqref{1.positive} and let, in addition,
\begin{equation}\label{1.fool}
|\gamma'(t)|+|\gamma(t)|\le C,\ \ t\ge0.
\end{equation}
Then, for every solution $y(t)$ of equation \eqref{1.non-lin}, the following dissipative estimates hold:
\begin{equation}\label{1.ddis}
E(t)\le C E(0)e^{-\alpha t}+C_*,
\end{equation}
where the positive constants $\alpha$ and $C_*$ are independent of $t$ and $E(0)$.
\end{prop}
\begin{proof}Although the heuristic arguments given above can be made rigorous, we prefer to verify \eqref{1.ddis} in a more straightforward way using the proper adaptation of the standard energy type estimates. Namely, we multiply equation \eqref{1.non-lin} by $y'(t)+\frac{2}{p+4}\gamma(t)y(t)$. Then, after elementary transformations, we get
\begin{multline}
\frac d{dt}\(E(t)+\frac2{p+4}\gamma(t)y'(t)y(t)\)+\frac{2(p+2)}{p+4}\gamma(t)E(t)=\\
=\frac2{p+2}\(\gamma'(t)-\gamma^2(t)\)y'(t)y(t).
\end{multline}
Let $\mathcal E(t):=E(t)+\frac2{p+4}\gamma(t)y'(t)y(t)$. Then using assumption \eqref{1.fool} and the fact that $p>0$ (analogously to \eqref{1.hold}), we deduce that
\begin{equation}\label{1.enequiv}
C_2(E(t)-1)\le\mathcal E(t)\le C_1 (E(t)+1)
\end{equation}
for some positive numbers $C_1$ and $C_2$ and
$$
\frac d{dt}\mathcal E(t)+\(\frac{2(p+2)}{p+4}\gamma(t)-\kappa\)\mathcal E(t)\le C_\kappa,
$$
where $\kappa>0$ is arbitrary. Integrating this inequality, we arrive at
\begin{equation}\label{1.gron}
\mathcal E(t)\le \mathcal E(0)e^{-\int_0^t\(2\frac{p+2}{p+4}\gamma(\tau)-\kappa\)\,d\tau}+
C_\kappa\int_0^te^{-\int_s^t\(2\frac{p+2}{p+4}\gamma(\tau)-\kappa\)\,d\tau}\,ds.
\end{equation}
According to assumptions \eqref{1.positive}, there exist $T>0$ and $\alpha>0$ such that
$$
2\frac{p+2}{p+4}\int_s^{s+nT}\gamma(\tau)\,d\tau\ge 2\alpha nT,\ \  s\ge0,\ \ n\in\mathbb N.
$$
Together with assumption \eqref{1.fool}, this gives
$$
2\frac{p+2}{p+4}\int_s^t\gamma(\tau)\,d\tau\ge 2\alpha(t-s)+C,\ \ t\ge s\ge0
$$
for some positive $C$ which is independent of $t$ and $s$. Fixing now $\kappa=\alpha$ and inserting this estimate t \eqref{1.gron}, we arrive at the desired estimate
$$
\mathcal E(t)\le C\mathcal E(0)e^{-\alpha t}+C_*
$$
which finishes the proof of the proposition.
\end{proof}
\subsection{Non-linear PDE. Key observation} We now turn to the model PDE
\begin{equation}\label{1.pde}
\Dt^2 u+\gamma(t)\Dt u-\Dx u+u|u|^p=0,\ x\in\Omega,\ \ u\big|_{\partial\Omega}=0
\end{equation}
in a bounded domain $\Omega$ of $\R^3$. The global well-posedness of this problem will be discussed in the next section and here we concentrate on the conditions for dissipativity. Similarly to the case of an ODE, formal multiplication of the equation by $\Dt u$ and integrating over $x\in\Omega$ give the energy identity
\begin{multline}
\frac d{dt}E(t):=\\=\frac d{dt}\(\frac12\|\Dt u\|^2_{L^2}+\frac12\|\Nx u\|^2_{L^2}+\frac1{p+2}\|u\|_{L^{p+2}}^{p+2}\)=-2\gamma(t)E_k(t),
\end{multline}
where
$$
E_k(t):=\frac12\|\Dt u\|^2_{L^2},\ \ E_p(t):=\frac12\|\Nx u\|^2_{L^2}+\frac1{p+2}\|u\|^{p+2}_{L^{p+2}}.
$$
Therefore, arguing as in the case of an ODE, we get
\begin{equation}\label{1.itr}
E(\eb)-E(0)\approx-2\( \frac{\<E_k\>}{\<E\>}\int_0^\eb\gamma(\tau)\,d\tau\)E(0)
\end{equation}
However, in contrast to the case of an ODE, the potential energy $E_p$ is no more homogeneous with respect to $u$ (due to the presence of the extra quadratic term $\frac12\|\Nx u\|^2_{L^2}$), so the ratio between the averaged kinetic and total energy is no more a constant, but may essentially depend on the trajectory considered. Indeed, multiplying equation \eqref{1.pde} by $u$ and integrating in $x$ and $t$, we get the following analogue of \eqref{1.eb}:
\begin{multline*}
\<\|u\|^{p+2}_{L^{p+2}}\>+\<\|\Nx u\|^2_{L^2}\>-2\<E_k\>=\\=-\<\gamma (\Dt u,u)\>+\frac{(\Dt u(0),u(0))-(\Dt u(\eb),u(\eb))}\eb,
\end{multline*}
where $(f,g)$ stands for the standard inner product in $L^2(\Omega)$. Thus,
\begin{equation}\label{1.bounds}
2\<E_k\>\approx \<\|\Nx u\|^2_{L^2}\>+\<\|u\|^{p+2}_{L^{p+2}}\>.
\end{equation}
Using that
$$
2\<E_p\>\le \<\|\Nx u\|^2_{L^2}\>+\<\|u\|^{p+2}_{L^{p+2}}\>\le (p+2)\<E_p\>,
$$
we finally arrive at
\begin{equation}\label{1.ebpde}
\frac12\ \lessapprox\  \frac{\<E_k\>}{\<E\>}\ \lessapprox \ \frac{p+2}{p+4}.
\end{equation}
Thus, when we iterate inequality \eqref{1.itr}, in the worst possible scenario the ratio $\frac{\<E_k\>}{\<E\>}$ will be close to $\frac12$ when $\gamma$ is positive and to $\frac{p+2}{p+4}>\frac12$ when $\gamma$ is negative. By this reason, the dissipativity condition \eqref{1.positive} should be naturally replaced by the stronger one
 \begin{equation}\label{1.Negative}
\liminf_{T\to\infty}\inf_{t\ge0}\frac1T\int_t^{t+T}\(\frac12\gamma(s)_+-\frac{p+2}{p+4}\gamma(s)_-\)ds>0
\end{equation}
which coincides with the assumption stated in Introduction.
\begin{rem} Actually, we do not know how to build up an explicit example showing that assumption \eqref{1.positive} is {\it not enough} for equation \eqref{1.pde} to be dissipative. On the other hand, we also do not know any mechanism to prevent the above mentioned worst scenario to appear. Indeed, the lower bound in \eqref{1.ebpde} is "attained" if the term $\<\|\Nx u\|_{L^2}^2\>$ in the averaged potential energy is dominating (i.e., if the solution oscillates rapidly in space and remains not very big). For the upper bound, we need large solution which oscillates not too fast in space (in this case the term $\<\|u\|^{p+2}_{L^{p+2}}\>$ will dominates). Such two different type of solutions can be constructed for the Hamiltonian case $\gamma=0$ (say, in the class of time periodic solutions using the perturbation technique). If we assume in addition that this Hamiltonian system is chaotic at any energy level $E$ then this worst scenario becomes indeed natural. By this reason, we conjecture that assumption \eqref{1.positive} is not enough for dissipativity and should be replaced by \eqref{1.Negative}.
\end{rem}

\section{Statement of the problem and dissipativity}\label{s2}
In this section, we recall  well-known facts on the existence and uniqueness of solutions for damped wave equation and a rigorous proof for dissipative estimates discussed in the previous section. Recall that we study the following damped wave equation:
\begin{equation}\label{2.wave}
\Dt^2u+\gamma(t)\Dt u-\Dx u+f(u)=g, \ u\big|_{\partial\Omega}=0,\ \ u\big|_{t=\tau}=u_\tau,\ \ \Dt u\big|_{t=\tau}=u'_\tau
\end{equation}
in a bounded domain $\Omega\subset\R^3$ with a smooth boundary. We assume that $g\in L^2(\Omega)$ and the nonlinearity $f\in C^1(\R,\R)$ has the following structure:
\begin{equation}\label{2.f}
f(u)=u|u|^p+f_0(u),\ \ \lim_{|u|\to\infty}\frac{|f_0'(u)|}{|u|^{p}}=0,\ \ f_0(0)=0
\end{equation}
for some $p>0$. Thus, the leading term of the nonlinearity is $u|u|^p$ exactly as in the previous section. Concerning the dissipation coefficient $\gamma$, we assume that it belongs to the uniformly local space $L^1_b(\R)$:
\begin{equation}\label{2.ul}
\gamma\in L^1_b(\R),\ \ \|\gamma\|_{L^1_b}:=\sup_{t\in\R}\|\gamma\|_{L^1((t,t+1))}<\infty
\end{equation}
and is {\it translation-compact} in it, i.e.:
\begin{equation}\label{2.trc}
\gamma\in L^1_{tr-c}(\R):=[C^\infty_b(\R)]_{L^1_b(\R)},
\end{equation}
where $[\cdot]$ stands for the closure. We recall that a function $\gamma\in L^1_b(\R)$ is translation compact if and only if it possesses a uniform  $L^1$-modulus of continuity:
\begin{equation}\label{2.mod}
\lim_{h\to0}\sup_{t\in\R}\int_t^{t+1}|\gamma(\tau+h)-\gamma(\tau)|\,d\tau=0,
\end{equation}
see \cite{CV02} for more details. In addition, we assume that the uniform analogue of  dissipativity condition \eqref{1.Negative} is satisfied:
 \begin{equation}\label{2.Negative}
\liminf_{T\to\infty}\inf_{t\in\R}\frac1T\int_t^{t+T}\(\frac12\gamma_+(s)-\frac{p+2}{p+4}\gamma_-(s)\)ds>0.
\end{equation}
As usual, we denote $\xi_u(t):=\{u(t),\Dt u(t)\}$ and introduce the energy space
$$
E:=[H^1_0(\Omega)\cap L^{p+2}(\Omega)]\times L^2(\Omega).
$$
As usual,
$H^1_0(\Omega)$ stands for the subspace of the Sobolev space $H^1(\Omega)$ with extra condition $u\big|_{\partial\Omega}=0$.
\par
We start with defining the  {\it energy} solutions of \eqref{2.wave}.
\begin{Def} A function $u(t)$, $t\ge\tau$ is a weak energy solution of problem \eqref{2.wave} if
$$
\xi_u\in L^\infty(\tau,\infty;E)
$$
and the equation is satisfied in the sense of distributions. The latter means that, for every test function $\varphi\in C_0^\infty((\tau,\infty)\times\Omega)$,
\begin{multline}
-\int_\R(\Dt u(t),\Dt\varphi(t))\,dt+\int_R\gamma(t)(\Dt u(t),\varphi(t))\,dt+\\+\int_\R(\Nx u(t),\Nx\varphi(t))\,dt+\int_\R(f(u(t)),\varphi(t))\,dt=\int_\R(g,\varphi(t))\,dt.
\end{multline}
Note that, since $\xi_u(t)\in E$, we have $f(u(t))\in L^q(\Omega)$, $q=\frac{p+2}{p+1}$. Thus, taking into the account that $\gamma\in L^1_b(\R)$, we get that the distributional derivative
$$
\Dt^2u\in L^1(\tau,\tau+T;H^{-1}(\Omega)+L^q(\Omega)),\ \ T>0
$$
and therefore $\Dt u\in C(\tau,\infty;H^{-1}(\Omega)+L^q(\Omega))$ which shows that the initial data for $\Dt u$ at $t=\tau$ is well-posed. The situation with the initial data for $u$ is simpler since we have $\Dt u\in L^\infty(\tau,\infty;L^2)$. The above arguments also imply in a standard way that the trajectory $\xi_u(t)$ is {\it continuous} in time in a weak topology of $E$:
$$
\xi_u\in C(\tau,\infty;E_w),
$$
see \cite{CV02} for more details.
\end{Def}
It is well-known that the weak energy solutions are well-posed if $p\le2$. For the case $2<p\le4$, well-posedness still holds in a slightly stronger class of solutions based on the so-called Strichartz estimates.
\begin{Def} A weak energy solution $u(t)$ is called Shatah-Struwe solution (SS-solution) if, in addition,
\begin{equation}\label{2.str}
u\in L^4(\tau,\tau+T;L^{12}(\Omega)),\ \ T\ge0.
\end{equation}
\end{Def}
We summarize the known results about the existence and uniqueness of solutions for \eqref{2.wave} in the following proposition.
\begin{prop} Let the functions $f$ and $\gamma$ satisfy the above assumptions and let $\xi_u(\tau):=\{u_\tau,u'_\tau\}\in E$. Then,
\par
1) There exists at least one weak energy solution of problem \eqref{2.wave} (no matter how big the growth exponent $p$ is).
\par
2) Let in addition $0\le p\le2$. Then the weak energy solution is unique, the function
$$
t\to\|\xi_u(t)\|_{\mathcal E}^2:=\frac12\|\Dt u(t)\|^2_{L^2}+\frac12\|\Nx u(t)\|^2_{L^2}+(F(u(t)),1)
$$
is absolutely continuous (here and below $F(u):=\int_0^uf(v)\,dv$) and the energy identity
$$
\frac d{dt}\|\xi_u(t)\|^2_{\mathcal E}=-\gamma(t)\|\Dt u(t)\|^2_{L^2}
$$
holds for almost all $t$.
\par
3) Let $0\le p\le 4$. Then there exists a unique SS-solution of problem \eqref{2.wave}. This solution also satisfies the energy identity in the above sense.
\end{prop}
\begin{rem} Indeed, the proof of the first statement is a standard application of the Galerkin approximation method and can be found, e.g., in \cite{CV02}. The second statement is also classical, see \cite{Temam, CV02} and references therein.
\par
The third statement is more recent and a bit more delicate. The local existence of SS-solutions follows using the perturbation arguments from the Strichartz estimate for the linear equation. Namely, let $V$ solves
$$
\Dt^2V-\Dx V=h(t),\ \ V\big|_{\partial\Omega}=0,\ \xi_v\big|_{t=0}=\xi_0
$$
with $\xi_0\in E:=H_0^1(\Omega)\times L^2(\Omega)$ and $h\in L^1(0,T;L^2(\Omega))$. Then,
\begin{equation}\label{2.s}
\|\xi_V\|_{C(0,T;E)}+\|V\|_{L^4(0,T;L^{12}(\Omega))}\le C_T(\|\xi_0\|_{E}+\|h\|_{L^1(0,T;L^2(\Omega))}),
\end{equation}
see \cite{BSogge} for more details.
\par
Here and below, we will use the fact that due to Sobolev's embedding $H^1\subset L^6$ and assumption $p\le4$, the term $L^{p+2}$ in the definition of the energy space $E$ is not necessary and can be omitted. We will use the following truncated energy norm
$$
\|\xi_u(t)\|_{E}^2: =\frac12\|\Dt u(t)\|^2_{L^2}+\frac12\|\Nx u(t)\|^2_{L^2}
$$
and the notation $\|\xi_u(t)\|^2_{\mathcal E}$ will be used for the full energy (including the $L^{p+2}$-norm).
\par
 For the sub-critical case $p<4$ the global solvability follows in a straightforward way from the local one and the so-called energy-to-Strichartz estimate for solutions \eqref{2.wave}
 \begin{equation}\label{2.es}
 \|u\|_{L^4(t,t+1;L^{12})}\le Q(\|\xi_u(t)\|_E)+Q(\|g\|_{L^2})
 \end{equation}
 for some monotone function $Q$ which is independent of $u$ and $t$. This estimate (which also follows from Strichartz estimate \eqref{2.s} via the perturbation arguments, see e.g.,\cite{KSZ}) is crucial since it allows to control the Strichartz norm of the solution through its energy norm. In particular, the dissipativity in the Strichartz norm will follow immediately if the dissipativity in the energy norm is established, we will utilize this fact in the next section.
\par
We also mention that, in the subcritical case $p<4$, the function $Q(z)$ is polynomial with respect to $z$:
\begin{equation}\label{2.polc}
Q(z)\le C_p(1+z^{N_p}),
\end{equation}
where the exponent $N_p$ may tend to infinity when $p\to4$.
\par
However, the energy-to-Strichartz estimate \eqref{2.es} is problematic in the critical case $p=4$. To the best of our knowledge it is proved for the periodic boundary conditions only and its validity for other boundary conditions is an open problem, see \cite{SZ}. In the critical case, the global existence of SS-solutions is verified using the so-called Pohozhaev-Morawetz inequality and related non-concentration estimates, see \cite{Grill,SS,Sogge08}. Note that, in contrast to the subcritical case, these arguments give only the global existence of an SS-solution without any quantitative bounds on its Strichartz norm. In particular, this norm may a priori grow uncontrollably as $t\to\infty$ and this prevents the applications to the attractor theory. Up to the moment, this problem is overcome in the autonomous case only, see \cite{KSZ}. By this reason, we will consider the attractor theory for equation \eqref{2.wave} for the subcritical case $p<4$ only (see \cite{Gin} for the uniqueness of energy solutions for $p<4$ in the whole space).
\par
Finally, the validity of energy identity for SS-solutions (and $p\le4$) is straightforward, see e.g., \cite{KSZ}. Note that the uniqueness of energy solutions for $p>2$ and SS-solutions for $p>4$ are not known yet.
\end{rem}
\begin{cor}\label{Cor2.pro} Let $0<p\le 4$ and the functions $\gamma$ and $f$ satisfy the above stated assumptions. Then equation \eqref{2.wave} generates a dynamical process $U_\gamma(t,\tau)$, $t\ge\tau$, in the energy space $E$ via
\begin{equation}\label{2.pro}
U_\gamma(t,\tau)\xi_\tau:=\xi_u(t),\ \ t\ge\tau,\  \tau\in\R,\ \xi_\tau\in E,
\end{equation}
where $\xi_u(t)=\{u(t),\Dt u(t)\}$ and $u(t)$ is a SS-solution of \eqref{2.wave} which corresponds to the initial data $\xi_u\big|_{t=\tau}=\xi_\tau$.
\end{cor}
We are now ready to state and prove the main result of this section.
\begin{theorem}\label{Th2.main} Let the functions $f$, $\gamma$ and $g$ satisfy the assumptions stated in the beginning of this section and let $p\le4$. Then, the following estimate holds for the SS-solutions of problem \eqref{2.wave}:
\begin{equation}\label{2.dis}
\|\xi_u(t)\|_{\mathcal E}\le C(1+\|\xi_u(\tau)\|_{\mathcal E})e^{-\alpha(t-\tau)}+C(1+\|g\|_{L^2}),\ \ t\ge\tau
\end{equation}
for some positive constants $C$ and $\alpha$ which are independent of $\tau\in\R$, $t\ge \tau$ and $\xi_u(\tau)$.
\end{theorem}
\begin{proof} We first utilize the fact that $\gamma$ is translation compact in $L^1_b(\R)$, so, for every $\eb>0$, we may find a function $\bar\gamma\in C^1_b(\R)$ such that
\begin{equation}\label{2.geb}
\|\gamma-\bar\gamma\|_{L^1_b}\le\eb.
\end{equation}
This function can be constructed using the standard mollifiers:
\begin{equation}\label{2.mol}
\bar\gamma(t):=\int_0^\infty k_\nu(s)\gamma(t-s)\,ds,\ \ k_\nu(s):=\nu^{-1}k(s/\nu)
\end{equation}
with $k\in C_0^\infty(\R_+)$ such that $\int_0^\infty k(s)\,ds=1$ and $\nu=\nu(\eb)$ is small enough.
\par
Note also that the function $\bar\gamma$ also satisfies \eqref{2.Negative} if $\eb>0$ is small enough.
Moreover, since $p\le4$, the SS-solution satisfies the energy identity
\begin{multline}\label{2.enn}
\frac d{dt}\(\frac12\|\Dt u\|^2_{L^2}+\frac12\|\Nx u\|^2_{L^2}+\frac1{p+2}\|u\|^{p+2}_{L^{p+2}}+(F_0(u(t)),1)\)+\\+\bar\gamma(t)\|\Dt u\|^2_{L^2}=-\tilde\gamma(t)\|\Dt u\|^2_{L^2},
\end{multline}
where $\tilde\gamma(t):=\gamma(t)-\bar\gamma(t)$ and $F_0(u):=\int_0^uf_0(v)\,dv$. At the next step, we multiply equation \eqref{2.wave} by $\frac12\bar\gamma_+(t)u-\frac2{p+4}\bar\gamma_-(t)u$ which gives
\begin{multline}\label{2.au}
\frac d{dt}\((\frac12\bar\gamma_+(t)-\frac2{p+4}\bar\gamma_-(t))(\Dt u,u)\)-\\-\(\frac12\bar\gamma_+(t)-\frac2{p+4}\bar\gamma_-(t)\)\|\Dt u\|^2_{L^2}+\\+\(\frac12\bar\gamma_+(t)-\frac2{p+4}\bar\gamma_-(t)\)\(\|\Nx u\|^2_{L^2}+\|u\|^{p+2}_{L^{p+2}}\)=\\=\(\frac12\bar\gamma_+(t)-\frac2{p+4}\bar\gamma_-(t)\)\((g,u)-(f_0(u),u)-\gamma(t)(\Dt u,u)\)+\\+\(\frac12\bar\gamma_+'(t)-\frac2{p+4}\bar\gamma_-'(t)\)(\Dt u,u):=H_u(t).
\end{multline}
Analogously to \eqref{1.bounds}, we have
\begin{multline*}
2\(\frac12\|\Nx u\|^2_{L^2}+\frac1{p+2}\|u\|^{p+2}_{L^{p+2}}\)\le\\\le\|\Nx u\|^2_{L^2}+\|u\|^{p+2}_{L^{p+2}}\le(p+2)\(\frac12\|\Nx u\|^2_{L^2}+\frac1{p+2}\|u\|^{p+2}_{L^{p+2}}\)
\end{multline*}
and using that $\bar\gamma_\pm(t)\ge0$ together with $\bar\gamma_+(t)\bar\gamma_-(t)\equiv0$, we get
\begin{multline*}
\(\frac12\bar\gamma_+(t)-\frac2{p+4}\bar\gamma_-(t)\)\(\|\Nx u\|^2_{L^2}+\|u\|^{p+2}_{L^{p+2}}\)\ge\\\ge 2\(\frac12\bar\gamma_+(t)-\frac{p+2}{p+4}\bar\gamma_-(t)\)E_p(t)
\end{multline*}
with $E_p(t):=\frac12\|\Nx u\|^2_{L^2}+\frac1{p+2}\|u\|^{p+2}_{L^{p+2}}$. Taking a sum of \eqref{2.au} and \eqref{2.enn} and using that $\bar\gamma(t)=\bar\gamma_+(t)-\gamma_-(t)$, we arrive at
\begin{multline}\label{2.ineq}
\frac d{dt}\mathcal E_u(t)+\\+2\(\frac12\bar\gamma_+(t)-\frac{p+2}{p+4}\bar\gamma_-(t)\)\(\frac12\|\Dt u\|^2_{L^2}+\frac12\|\Nx u\|^2_{L^2}+\frac1{p+2}\|u\|^{p+2}_{L^{p+2}}\)\le\\\le H_u(t)-\tilde\gamma(t)\|\Dt u\|^2_{L^2},
\end{multline}
where
\begin{multline*}
\mathcal E_u(t):=\frac12\|\Dt u\|^2_{L^2}+\frac12\|\Nx u\|^2_{L^2}+\frac1{p+2}\|u\|^{p+2}_{L^{p+2}}+\\+(F_0(u(t)),1)+
(\frac12\bar\gamma_+(t)-\frac2{p+4}\bar\gamma_-(t))(\Dt u,u).
\end{multline*}
Since $\bar\gamma\in C^1_b(\R)$, $p>0$ and $F_0(u)$ is subordinated by $u|u|^p$, we have
\begin{multline}
\frac12\(\frac12\|\Dt u\|^2_{L^2}+\frac12\|\Nx u\|^2_{L^2}+\frac1{p+2}\|u\|^{p+2}_{L^{p+2}}\)-C_\eb\le\mathcal E_u\le\\\le 2\(\frac12\|\Dt u\|^2_{L^2}+\frac12\|\Nx u\|^2_{L^2}+\frac1{p+2}\|u\|^{p+2}_{L^{p+2}}\)+C_\eb.
\end{multline}
Analogously, using again that $p>0$ and that $f_0$ is subordinated to $u|u|^p$, we may estimate the term $H_u(t)$ as follows:
\begin{multline}
H_u(t)\le C_\eb\(|(|g|,|u|)+|(f_0(u)),u)|+(1+|\gamma(t))|(|\Dt u|,|u|)\)\le\\\le \kappa\(|\gamma(t)|+1\)\mathcal E_u(t)+C_{\kappa,\eb}\(1+|\gamma(t)|+\|g\|^2_{L^2}\),
\end{multline}
where $\eb>0$ and $\kappa>0$ can be taken arbitrarily small. Inserting the obtained estimates to \eqref{2.ineq}, we finally arrive at
\begin{equation}\label{2.good}
\frac d{dt}\mathcal E(t)+\beta_\eb(t)\mathcal E(t)\le C_{\kappa,\eb}\(1+\|g\|^2_{L^2}+|\gamma(t)|\),
\end{equation}
where
\begin{equation}\label{2.beta}
\beta_\eb(t):=2\(\frac12\bar\gamma_+(t)-\frac{p+2}{p+4}\bar\gamma_-(t)-|\tilde\gamma(t)|-
\kappa|\gamma(t)|-\kappa\).
\end{equation}
Due to the fact that $\gamma\in L^1_b(\R)$ satisfies \eqref{2.Negative} and $\tilde \gamma$ is of order $\eb$ (see \eqref{2.geb}), we may fix positive constants $\eb$ and $\kappa$ in such a way that assumption \eqref{2.Negative} will be satisfied for the for $\beta_\eb(t)$ as well. Namely,
$$
\liminf_{T\to\infty}\frac1T\inf_{t\in\R}\int_{t}^{t+T}\beta_\eb(s)\,ds>0.
$$
Analogously to the proof of Proposition \ref{Prop1.dis}, this gives
\begin{equation}\label{2.betag}
\int_s^t\beta_\eb(\tau)\,d\tau\ge 2\alpha(t-\tau)+C,\ \ t\ge s,
\end{equation}
where positive constants $\alpha$ and $C$ are independent of $t$ and $\tau$. Integrating inequality \eqref{2.good} in time, we get
\begin{equation}\label{2.Gron}
\mathcal E(t)\le \mathcal E(\tau)e^{-\int_\tau^t\beta_\eb(s)\,ds}+C\int_\tau^t(1+\|g\|^2+|\gamma(s)|)e^{-\int_s^t\beta_\eb(l)\,dl}\,ds
\end{equation}
which together with \eqref{2.betag} gives the desired estimate \eqref{2.dis} and finishes the proof of the theorem.
\end{proof}
\begin{rem} As we can see from the proof, neither the fact that $\Omega$ is three-dimensional, nor the conditions on the growth rate $p$ are essentially used in the derivation of the dissipative estimate. Actually, these assumptions are posed in order to have global well-posedness and the validity of the energy identity only. Thus, both of them can be removed, but in this case, we will be unable to verify the dissipative estimate for {\it all} weak energy (or Strichartz) solutions and may only claim that for every initial data from the energy space, there exists a solution which satisfies the above dissipative estimate.
\end{rem}
\begin{cor}\label{Cor2.abs} Let the assumptions of Theorem \ref{Th2.main} hold and let, in addition, $p<4$. Then, in addition to \eqref{2.dis}, we also have the dissipativity of the Strichartz norm. Namely, for any SS-solution $u(t)$ of equation \eqref{2.wave}, the following estimate holds:
\begin{equation}\label{2.stdis}
\int_{t-1}^t\|u(s)\|^4_{L^{12}}\,ds\le Q(\|\xi_u(\tau)\|_E)e^{-\alpha(t-\tau)}+Q(\|g\|_{L^2}),\ \ t\ge\tau+1
\end{equation}
for some positive $\alpha$ and monotone increasing function $Q$.
\end{cor}
Indeed, estimate \eqref{2.stdis} follows from  \eqref{2.dis} and the energy-to-Strichartz estimate \eqref{2.es}.
\begin{cor}\label{Cor2.lip} Let the assumptions of Corollary \ref{Cor2.abs} hold. Then the dynamical process $U(t,\tau)$ is Lipschitz continuous uniformly on bounded sets of the energy space $E$. Namely, for any two SS-solutions $u(t)$ and $v(t)$ of equation \eqref{2.wave}, the following estimate holds:
\begin{multline}\label{2.lip}
\|\Dt u(t)-\Dt v(t)\|^2_{L^2}+\|\Nx u(t)-\Nx v(t)\|^2_{L^2}\le\\\le Ce^{L(t-\tau)}\(\|\Dt u(\tau)-\Dt v(\tau)\|^2_{L^2}+\|\Nx u(\tau)-\Nx v(\tau)\|^2_{L^2}\),\ \ t\ge\tau,
\end{multline}
where the constant $L$ depends only on $\gamma$, $f$ and $g$ and the constant $C$ depends also on the energy norms of $\xi_u(\tau)$ and $\xi_v(\tau)$.
\end{cor}
\begin{proof} Let $w(t):=u(t)-v(t)$. Then this function solves
\begin{equation}\label{2.minus}
\Dt^2 w+\gamma(t)\Dt w-\Dx w+l(t)w=0, \ \ \xi_w\big|_{t=\tau}=\xi_u(\tau)-\xi_v(\tau),
\end{equation}
where $l(t):=\int_0^1f'(su(t)+(1-s)v(t))\,ds$. Since both $u$ and $v$ satisfy energy identity, we may multiply the above equation by $\Dt w$ and integrate over $x\in\Omega$ to get
$$
\frac12\frac d{dt}\(\|\Dt w\|^2_{L^2}+\|\Nx w\|^2_{L^2}\)=-\gamma(t)\|\Dt w\|^2_{L^2}-(l(t)w,\Dt w)
$$
and we only need to estimate the last term in the right-hand side of this equality. Since $f'(u)$ grows as $|u|^p$ with $p<4$, we have
\begin{multline}
(l(t)w,\Dt w)\le\|l(t)\|_{L^3}\|w\|_{L^6}\|\Dt w\|_{L^2}\le\\\le C(1+\|u(t)\|^4_{L^{12}}+\|v\|^4_{L^{12}})\(\|\Dt w\|^2+\|\Nx w\|^2\)
\end{multline}
and, therefore,
\begin{multline}
\frac12\frac d{dt}\(\|\Dt w\|^2_{L^2}+\|\Nx w\|^2_{L^2}\)\le \\\le \(|\gamma(t)|+C(1+\|u(t)\|^4_{L^{12}}+\|v\|^4_{L^{12}})\)\(\|\Dt w\|^2_{L^2}+\|\Nx w\|^2_{L^2}\).
\end{multline}
Integrating this inequality and using estimate \eqref{2.stdis} together with the fact that $\gamma\in L^1_b(\R)$, we arrive at \eqref{2.lip} and finish the proof of the corollary.
\end{proof}

\section{Asymptotic regularity}\label{s3}
In this section we will prove that, similarly to the case of constant dissipation rate, there is an asymptotic smoothing property for Strichartz solutions of \eqref{2.wave}. Namely, we will split a solution $u(t)$ of this equation in two parts
$$
u(t)=v(t)+w(t),
$$
where $w(t)$ is more regular and $v(t)$ is exponentially decaying. However, in contrast to the standard situation, we cannot take $v$ as a solution of the linear equation with $f=g=0$ (as explained in Section \ref{s1}, this solution may be unstable) and should proceed in a more delicate way. The key technical tool for this is the following proposition.
\begin{prop}\label{Prop3.lin} Let the function $\gamma(t)$ satisfy \eqref{2.ul}, \eqref{2.trc} and \eqref{2.Negative} with $p=0$. Assume also that $h\in L^1_{loc}(\R,L^2(\Omega))$. Then, there exists a positive constant $L=L(\gamma)$ such that the energy solution $v(t)$ of equation
\begin{equation}\label{3.Llin}
\Dt^2 v+\gamma(t)\Dt v-\Dx v+Lv=h,\ \ \xi_v\big|_{t=\tau}=\xi_\tau,\ \ v\big|_{\partial\Omega}=0
\end{equation}
satisfies the dissipative estimate
\begin{multline}\label{4.lindis}
\|v(t)\|^2_{H^1}+\|\Dt v(t)\|^2_{L^2}+\(\int_{\max\{\tau,t-1\}}^t\|u(s)\|^4_{L^{12}}\,ds\)^{1/2}\le\\\le C\(\|v(t)\|^2_{H^1}+\|\Dt v(t)\|^2_{L^2}\)e^{-\alpha(t-\tau)}+
C\(\int_\tau^te^{-\alpha(t-s)}\|h(s)\|_{L^2}\,ds\)^2,
\end{multline}
where the positive constants $C$ and $\alpha$ are independent of $t$, $\tau$ and $\xi_\tau$.
\end{prop}
\begin{proof} We first note that it is sufficient to prove \eqref{4.lindis} for the energy norm only. The estimate for the Strichartz norm will then follow from estimate \eqref{2.s}. Moreover, it is sufficient to verify estimate \eqref{4.lindis} for the case $h=0$ only, the general case follows then from the variation of constants formula.
\par
Analogously to the proof of Theorem \ref{Th2.main}, we split the function $\gamma(t):=\bar\gamma(t)+\tilde \gamma(t)$, where $\bar\gamma\in C^1_b(\R)$ and $\tilde\gamma$ satisfies \eqref{2.geb}. Multiplying equation \eqref{3.Llin} by $\Dt u+\frac12\bar\gamma u$, after the elementary transformations we get
\begin{equation}\label{3.grr}
\frac d{dt}\mathcal E_v(t)+\bar\gamma(t)\mathcal E_v(t)=H_v(t)
\end{equation}
where
$$
\mathcal E_v(t):=\frac12\|\Dt v\|^2_{L^2}+\frac12\|\Nx v\|^2_{L^2}+\frac L2\|v\|^2_{L^2}+\frac12\bar\gamma(t)(\Dt v,v)
$$
and
$$
H_v(t):=\frac12\bar\gamma'(t)(\Dt v,v)+\frac12\bar\gamma^2(t)(\Dt v,v)-\tilde\gamma(t)\|\Dt v\|^2_{L^2}.
$$
Since $\bar\gamma$ is globally bounded, we have
\begin{multline}
\frac12\(\frac12\|\Dt v\|^2_{L^2}+\frac12\|\Nx v\|^2_{L^2}+\frac L2\|v\|^2_{L^2}\)\le\\\le \mathcal E_v(t)\le 2\(\frac12\|\Dt v\|^2_{L^2}+\frac12\|\Nx v\|^2_{L^2}+\frac L2\|v\|^2_{L^2}\)
\end{multline}
if $L>L_0:=L_0(\eb)$. Analogously, since $\bar\gamma$ and $\bar\gamma'$ are bounded, for every $\kappa>0$, there exist $L_0=L_0(\kappa,\eb)$ such that
$$
H_v(t)\le \kappa\mathcal E_v(t)+2|\tilde\gamma(t)| \mathcal E_v(t)
$$
and therefore,
\begin{equation}\label{3.grrr}
\frac d{dt}\mathcal E_v(t)+(\bar\gamma(t)-\kappa-2|\tilde\gamma(t)|)\mathcal E_v(t)\le0
\end{equation}

if $L>L_0(\kappa,\eb)$. Fixing now $\kappa>0$ and $\eb>0$ small enough that
$$
\liminf_{T\to\infty}\frac1{T}\inf_{\tau\in\R}\int_\tau^{\tau+T}\(\bar\gamma(t)-\kappa-2|\tilde\gamma(t)|\)\,dt>0
$$
and integrating \eqref{3.grrr}, we get the desired estimate for the energy norm and finish the proof of the proposition.
\end{proof}
We are now ready to state and prove the main result of this section on the existence of a smooth exponentially absorbing set for the dynamical process associated with equation \eqref{2.wave}. Unfortunately, we cannot do this for the critical case of quintic nonlinearity $f$, so we have to impose the sub-criticality assumption:
\begin{equation}\label{3.subcrit}
p<4.
\end{equation}
\begin{theorem}\label{Th3.main} Let the assumptions of Theorem \ref{Th2.main} hold and let the nonlinearity $f$ be subcritical ($p<4$). Then, if $R=R(\gamma,f,g)$ is large enough, the closed ball $\mathcal B_R^1$ of radius $R$ in the space
$$
E^1:=[H^2(\Omega)\cap H^1_0(\Omega)]\times H^1_0(\Omega)
$$
is a uniformly attracting set for the dynamical process $U_\gamma(t,\tau)$ associated with equation \eqref{2.wave}. Namely, there are positive constant $\alpha$ and a monotone increasing function $Q$ such that, for any bounded set $B\subset E$,
\begin{equation}\label{3.expattr}
\operatorname{dist}_E(U_\gamma(t,\tau)B,\mathcal B^1_R)\le Q(\|B\|_E)e^{-\alpha(t-\tau)},\ \ t\ge\tau,
\end{equation}
holds uniformly with respect to $\tau\in\R$. Here and below
$$
\operatorname{dist}_E(U,V):=\sup_{u\in U}\inf_{v\in V}\|u-v\|_E
$$
is a non-symmetric Hausdorff distance between sets $U$ and $V$ in a space $E$.
\end{theorem}
\begin{proof} We will use the standard bootstrapping arguments together with transitivity of exponential attraction (see e.g.,\cite{MirZel}). Let $G=G(x)$ solve the linear elliptic problem
$$
-\Dx G=g,\ \ G\big|_{\partial\Omega}=0.
$$
Then, due to the elliptic regularity, $G\in H^2(\Omega)\cap H^1_0(\Omega)$ and
$$
\|G\|_{H^2}\le C\|g\|_{L^2}.
$$
Let $\xi_u(t):=U_\gamma(t,\tau)\xi_\tau$ be a solution of equation \eqref{2.wave}. Then, without loss of generality, we may assume that $\xi_\tau$ belongs to the uniformly absorbing ball $\mathcal B^0_R$ in the energy space $E$. Indeed, such a ball exists due to estimate \eqref{2.dis}. Thus, using also \eqref{2.stdis}, we may assume without loss of generality that
\begin{equation}\label{3.abs}
\|\xi_u(t)\|^2_E+\int_t^{t+1}\|u(s)\|^4_{L^{12}}\,ds\le Q(\|g\|_{L^2}).
\end{equation}
Let $v(t)$ solve the linear problem
\begin{equation}\label{3.lind}
\Dt^2 v+\gamma(t)\Dt v-\Dx v+Lv=0,\ \ \xi_v\big|_{t=\tau}=\{u(\tau)-G,\Dt u(\tau)\},
\end{equation}
where $L>0$ is such that conditions of Proposition \ref{Prop3.lin} are satisfied,
and the remainder $w(t)$ satisfies
\begin{multline}\label{3.rem}
\Dt^2 w+\gamma(t)\Dt w-\Dx w+Lw=\\=-f(u(t))+L(u(t)-G):=h_u(t),\ \ \xi_w\big|_{t=\tau}=0.
\end{multline}
Then, obviously, $u(t)=v(t)+G+w(t)$. Moreover, from Proposition \ref{Prop3.lin}, we have
\begin{equation}\label{3.v}
\|\Dt v(t)\|^2_E+\|\Nx v(t)\|^2_E\le Ce^{-\alpha (t-\tau)},
\end{equation}
where the positive constants $C$ and $\alpha$ are independent of $\xi_\tau\in \mathcal B^0_R$ and $t\ge\tau$.  In order to get the estimate for smoother component $w$, we use the fact that $|f'(u)|\le C(1+|u|^p)$ with $p<4$. Indeed, by interpolation,
$$
\|f(u)\|_{L^2}\le C(1+\|u\|_{L^{10}}^5)\le C(1+\|u\|^4_{L^{12}}\|u\|_{L^6})\le C_1(1+\|u\|^4_{L^{12}})
$$
and, therefore,
$$
\|f(u)\|_{L^1(t,t+1;L^2)}\le C,\ \ t\ge\tau.
$$
Analogously, using the H\"older inequality, we get
\begin{multline}
\|\Nx f(u)\|_{L^\kappa}=\|f'(u)\Nx u\|_{L^\kappa}\le\\\le C\|(1+|u|^p)|\Nx u|\|_{L^\kappa}\le C(1+\|u\|_{L^{12}}^p)\|\Nx u\|_{L^2},
\end{multline}
where $\frac1\kappa=\frac12+\frac{p}{12}$. Thus, we have
$$
\|f(u)\|_{L^1(t,t+1;W^{1,\kappa})}\le C,\ \ t\ge\tau.
$$
Using the embedding $W^{1,\kappa}\subset H^\beta$ where $\frac12=\frac1\kappa-\frac{1-\beta}3$, i.e.,
$$
\beta=\beta_1:=1-\frac p4>0,
$$
we arrive at the estimate
\begin{equation}\label{3.s1}
\|f(u)\|_{L^1(t,t+1;H^\beta)}\le C,\ \ t\ge\tau.
\end{equation}
Applying now the operator $(-\Dx)^{\beta/2}$ to both sides of equation \eqref{3.rem} and using Proposition \ref{Prop3.lin}, we get
\begin{equation}\label{3.w}
\|\Dt w(t)\|_{H^\beta}^2+\|\Nx w(t)\|^2_{H^{\beta}}\le C,
\end{equation}
where we have also implicitly used that
$$
\|L(u-G)\|_{L^1(t,t+1;H^1)}\le C,\ \ t\ge\tau.
$$
Estimates \eqref{3.v} and \eqref{3.w} show that the ball $\mathcal B^\beta_R$ in higher energy space
$$
E^\beta:=[H^{\beta+1}\cap H^1_0]\times H^\beta
$$
is a uniformly exponentially attracting set for the dynamical process $U_\gamma(t,\tau)$. Thus, the first (in a sense most difficult) step is completed.
\par
To initiate the next step of bootstrapping, we note that if we take $\xi_u(\tau)\in \mathcal B_R^\beta$ from the very beginning, we may apply the operator $(-\Dx)^{\beta/2}$ for the equation for $v(t)$ as well and get
$$
\|\Dt v(t)\|^2_{H^\beta}+\|\Nx v(t)\|^2_{H^\beta}\le Ce^{-\alpha(t-\tau)},\ \ t\ge\tau
$$
which together with \eqref{3.w} shows that the dynamical process $U_\gamma(t,\tau)$ is well defined and dissipative in the higher energy space $E^\beta$ as well. In particular, we now have the control of $\Nx u(t)$ not only in $L^2$, but in more regular space $H^\beta$. From Sobolev's embedding this gives the control of $\Nx u$ in the space $L^q$ where $\frac1q=\frac12-\frac\beta3$. In turn, arguing as before, but using this better control for the term $\Nx u$, we improve estimate \eqref{3.s1}
\begin{equation}
\|f(u)\|_{L^1(t,t+1;H^{\beta_1})}\le C,
\end{equation}
where
\begin{equation}\label{3.beta}
\beta_1=1-\frac p4+\beta.
\end{equation}
This gives the analogue of estimate \eqref{3.w} with $\beta$ replaced by $\beta_1>\beta$. Thus, the dynamical process $U_\gamma(t,\tau)$ on $\mathcal B_R^\beta$ has an exponentially attracting ball $\mathcal B^{\beta_1}_R$ in the space $E^{\beta_1}$ if $R$ is large enough. Moreover, since $U(t,\tau)$ is globally Lipschitz on $\mathcal B^0_R$ (see Corollary \ref{Cor2.lip}), the transitivity of exponential attraction gives that $\mathcal B^{\beta_1}_R$ is a uniformly exponentially attracting set for $U(t,\tau)$ in the initial energy space as well.
\par
Finally, iterating the above procedure, we get the exponentially attracting ball in the space $E^1$ which corresponds to $\beta=1$. Indeed, from \eqref{3.beta} we conclude that $\beta_n=n(1-\frac p4)$ and this guarantees that we reach the value $\beta=1$ in finitely many steps. Thus, the theorem is proved.
\end{proof}

\section{Attractors}\label{s4}
The aim of this section is to apply the results obtained above for constructing global and exponential attractors for the dynamical process $U_\gamma(t,\tau)$ associated with equation \eqref{2.wave}. The assumptions on the non-autonomous symbol $\gamma(t)$ posed in Section \ref{s2} are well adapted for usage the so-called {\it uniform} attractor. In order to build up such an object, following the general scheme (see \cite{CV02} for the details), we need to consider not only equation \eqref{2.wave}, but also all time shifts of this equation together with their limit in the proper topology. Namely, let us consider a hull $\mathcal H(\gamma)$ of the initial symbol $\gamma$ defined as follows:
\begin{equation}\label{4.hull}
\mathcal H(\gamma):=\big[T_h\gamma,\ h\in\R]_{L^1_{loc}(\R)},\ \ (T_h\gamma)(t):=\gamma(t+h).
\end{equation}
Then, obviously,
$$
T_h\mathcal H(\gamma)=\mathcal H(\gamma).
$$
Moreover, since $\gamma$ is translation-compact, the hull $\mathcal H(\gamma)$ is a compact set in $L^1_{loc}(\R)$, see \cite{CV02} for details.

From now on, we endow the hull $\mathcal H(\gamma)$ by the $L^1_{loc}(\R)$-topology and will not consider other topologies on this hull. For every $\eta\in\mathcal H(\gamma)$ let us consider the associated wave equation
\begin{equation}\label{4.wave}
\Dt^2u+\eta(t)\Dt u-\Dx u+f(u)=g,\  u\big|_{\partial\Omega}=0,\ \ \xi_u\big|_{t=0}=\xi_\tau.
\end{equation}
We denote by $U_\eta(t,\tau):E\to E$, $t\ge\tau$, the solution operator of this equation.
Then, as not difficult to see that \eqref{2.Negative}, \eqref{2.ul} and \eqref{2.geb} hold uniformly with
respect to $\eta \in\mathcal H(\gamma)$ and, therefore, all of the estimates obtained in Sections
\ref{s2} and \ref{s3} also hold uniformly with respect to $\eta\in\mathcal H(\gamma)$. We also note that the
family of processes $\{U_\eta(t,\tau),\, \eta\in\mathcal H(\gamma)\}$ possesses the so-called translation identity
$$
U_\eta(t+h,\tau+h)= U_{T_h\eta}(t,\tau)
$$
which in turn allows us to reduce the non-autonomous dynamical system considered to the autonomous semigroup $\mathbb S(t)$ acting on the extended phase space $\mathbb E:= E\times\mathcal H(\gamma)$ via
\begin{equation}\label{4.ext}
\mathbb S(t)\{\xi,\eta\}:=\{U_\eta(t,0)\xi, T_t\eta\},
\end{equation}
see \cite{CV02} for more details. Thus, we may define a global attractor $\mathbb A$ for the semigroup $\mathbb S(t)$ in the extended phase space $\mathbb E$.
\begin{Def}  A set $\mathbb A$ is a global attractor for the semigroup $\mathbb S(t): \mathbb E\to\mathbb E$ if
\par
a) $\mathbb A$ is compact in $\mathbb E$;
\par
b) It is strictly invariant: $\mathbb S(t)\mathbb A = \mathbb A$ for all $t\ge0$;
\par
c) It attracts bounded subsets of $\mathbb E$ as $t\to\infty$, i.e., for any bounded $\mathbb B\subset \mathbb E$ and any neighbourhood $\mathcal O(\mathbb A)$ there exists $T = T(\mathbb B, \mathcal O)$ such that
\begin{equation}\label{4.attr}
\mathbb S(t)\mathbb B\subset\mathcal O(\mathbb A),\ t\ge T.
\end{equation}
If the attractor $\mathbb A\subset\mathbb E$ exists, its projection
$\mathcal A_{un} := \Pi_1\mathbb A\subset E$ to the first component of the Cartesian product is called a uniform attractor for the family $U_\eta(t,\tau)$, $\eta\in\mathcal H(\gamma)$ of the dynamical processes associated with equation \eqref{4.wave}.
\end{Def}
The existence of a global/uniform attractor is usually verified using the
following standard result, see e.g., \cite{CV02} for more details.
\begin{prop}\label{Prop4.abs} Let the semigroup $\mathbb S(t):\mathbb E\to\mathbb E$ is
\par
a) Continuous as a map from $\Bbb E$ to $\Bbb E$ for every fixed $t\ge0$;
\par
b) Possesses a compact attracting set $\mathcal B\subset\mathbb E$, i.e., for every bounded
$\mathbb B\subset\mathbb E$
and every neighbourhood $\mathcal O(\mathbb B)$ there exists $T = T(\mathbb B, \mathcal O)$ such that \eqref{4.attr} is satisfied (with $\mathbb A$ replaced by $\mathcal B$).
\par
Then the semigroup $\mathbb S(t)$ possesses a global attractor $\mathbb A$. Moreover, if the semigroup $\mathbb S(t)$ is an extended semigroup for the family of processes $U_\eta(t,\tau):E\to E$, $\eta\in\mathcal H(\gamma)$, then the uniform attractor $\mathcal A_{un}=\Pi_1\mathbb A$ possesses the following description:
\begin{equation}\label{4.str}
\mathcal A_{un} =\cup_{\eta\in\mathcal H(\gamma)}\mathcal K_\eta\big|_{t=0},
\end{equation}
where
\begin{equation}
\mathcal K_\eta:= \{\xi_u\in B(\R,E):\ \xi_u(t)= U_\eta(t,\tau)\xi_u(\tau),\,\tau\in\R,\, t\ge\tau\}
\end{equation}
is the set of all complete (defined for all $t\in\R$) bounded solutions of equation \eqref{4.wave}
with fixed symbol $\eta$ (=the so-called kernel of equation \eqref{4.wave} in the
terminology of Chepyzhov and Vishik, see \cite{CV02}).
\end{prop}
Using the estimates obtained in previous two sections, we get the following result.

\begin{theorem}\label{Th4.unattr} Let the assumptions of Theorem \ref{Th3.main} hold. Then the family
$U_\eta(t,\tau)$ of dynamical processes associated with wave equation \eqref{4.wave} possesses
a uniform attractor $\mathcal A_{un}$ which is a bounded set of $E^1$. Moreover, this
attractor is generated by all complete bounded trajectories of equations \eqref{4.wave},
i.e., the representation formula \eqref{4.str} holds.
\end{theorem}
\begin{proof} Indeed, the continuity of the extended semigroup $\Bbb S(t)$ can be verified
exactly as in Corollary \ref{Cor2.lip} (we left the proof of continuity
with respect to
the symbol $\eta\in\mathcal H(\gamma)$ to the reader). Moreover, estimate \eqref{3.expattr} shows that
the set
$\mathbb B:= \mathcal B_R^1\times\mathcal H(\gamma)$ is a compact attracting set for the semigroup $\mathbb S(t)$. Here we have also used that the embedding $E^1\subset E$ is compact and the hull $\mathcal H(\gamma)$ endowed by the
$L^1_{loc}(\R)$ topology is also compact. Thus, all assumptions of Proposition \ref{Prop4.abs} are verified and, therefore, the existence of a uniform attractor $\mathcal A_{un}$ is also verified and the theorem is proved.
\end{proof}
\begin{rem}
There is an alternative equivalent (at least in the case where
$\mathbb S(t)$ is continuous, see \cite{CV02}) definition of a uniform attractor which does
not refer explicitly to the reduction to autonomous extended semigroup. Namely, the set $\mathcal A_{un}\subset E$ is a uniform attractor for the dynamical process $U_\gamma(t,\tau): E\to E$ if
\par
1) $\mathcal A_{un}$ is compact in $E$;
\par
2) It attracts bounded sets $B\subset E$ uniformly with respect to $\tau\in\R$. Namely, for every bounded $B\subset E$ and every neighbourhood $\mathcal O(\mathcal A_{un})$ there exists $T=T(B, \mathcal O)$ such that
$$
U_\gamma(t,\tau)B\subset \mathcal O(\mathcal A_{un}),\text{ if } t-\tau\ge T;
$$
3) $\mathcal A_{un}$ is a minimal set which satisfies properties 1) and 2).
\par
However, in order to get the representation formula \eqref{4.str}, we need in any case to introduce the hull of the non-autonomous symbol $\gamma$. Since this representation formula is crucial for the attractors theory, we prefer to introduce the uniform attractor using the extended semigroup from the very beginning.
\end{rem}
\begin{rem} For every $\eta\in\mathcal H(\gamma)$, we may define the so-called kernel sections
\begin{equation}
\mathcal K_\eta(\tau):= \mathcal K_\eta\big|_{t=\tau}.
\end{equation}
Then, as not difficult to see, the sets $\mathcal K_\eta(\tau)$ are compact and strictly invariant, i.e., $\mathcal K_\eta(t)= U_\eta(t,\tau)\mathcal K_\eta(\tau)$, $t\ge\tau\in\R$.
\par
Moreover, as proved in \cite{CV94,CV02}, for every fixed $\eta$, these kernel sections possess the so-called pullback attraction property, i.e., for any bounded set $B\subset E$, any fixed $t\in\R$ and any neighbourhood $\mathcal O(\mathcal K_\eta(t))$, there exists $T = T (t, B, \mathcal O)$ such that
\begin{equation}\label{4.pull}
 U_\eta(t,\tau)B\subset O(\mathcal K_\eta(t))\ \text{ if } t-\tau>T.
\end{equation}
By this reason, the time dependent family $t\to\mathcal K_\eta(t)$ is often called a pullback
attractor associated with equation \eqref{4.wave}, see \cite{Kloeden,Carvalho2012} and references therein. Note that, in contrast to the case of a uniform attractor, the rate of convergence in \eqref{4.pull} is not uniform with respect to $t\in\R$ and by this reason forward in time convergence may fail.
\end{rem}
We now turn to exponential attractors. This concept has been introduced
in \cite{EFNT} for the autonomous case in order to overcome the major drawback of global attractors, namely, the fact that the convergence to a global attractor may be arbitrarily slow and there is no way to control this rate in terms of physical parameters of the system considered. This leads to sensitivity of the attractor to perturbation and makes it in a sense unobservable in experiments. Roughly speaking, the idea of an exponential attractor is to add some extra points to the global attractor in such a way that, on the one hand, the rate of attraction to this object becomes exponential and controllable and, on the other hand, the size of the new attractor does not grows drastically, for instance, it should remain finite dimensional. The price to pay is that an exponential attractor is only positively invariant and as a
consequence, it is not unique, see \cite{EMZ,MirZel} and references therein.
\par
The situation is more delicate when non-autonomous equations are considered since new essential drawbacks of global attractors come into play.
Indeed, a uniform attractor is usually huge (infinite-dimensional) even when the ”real” attractor consists of the only exponentially stable trajectory and kernel sections (= pullback attractors) do not attract in a natural sense (forward in time). All these drawbacks can be overcome using the concept of a non-autonomous exponential attractor which is, on the one hand, remains finite-dimensional and, on the other hand, not only pullback but also forward and uniform
in time exponentially attracting. We use below the construction given in \cite{EMZ}.
\begin{Def} Let $U(t,\tau): E\to E$ be a dynamical process in a B-space $E$. Then, the family of sets $\mathcal M(t)\subset E$, $t\in\R$, is a non-autonomous exponential attractor for $U(t,\tau)$ if
\par
1) The sets $\mathcal M(t)$ are compact in $E$;
\par
2) They are positively semi-invariant: $U(t,\tau)\mathcal M(\tau)\subset\mathcal M(t)$;
\par
3) The fractal dimension of $\mathcal M(t)$ is finite and uniformly bounded:
$$
\dim_f (\mathcal M(t),E)\ge C<\infty
$$
for all $t\in\R$;
\par
4) They are uniformly exponentially attracting, i.e., there exists a positive constant $\alpha$ and a monotone increasing function $Q$ such that
$$
\dist_H(U(t,\tau)B, \mathcal M(t))\le Q(\|B\|_E)e^{-\alpha (t-\tau)}
$$
uniformly with respect to $t\ge\tau$ and $\tau\in\R$.
\end{Def}
We are now ready to state and prove the main result of this section.

\begin{theorem}\label{Th4.main} Let the assumptions of Theorem \ref{Th3.main} hold and let, in addition, $f\in C^2(\R)$ and
\begin{equation}\label{4.gsm}
\gamma\in L^{1+\eb}_b(\R)
\end{equation}
for some positive $\eb$. Then, the dynamical process $U_\gamma(t,\tau): E\to E$ associated with wave equation \eqref{2.wave} possesses a non-autonomous exponential attractor $t\to\mathcal M(t)=\mathcal M_\gamma(t)$ which is uniformly bounded in $E^1$ and the following H\"older continuity in time holds:
\begin{equation}\label{4.hld}
\dist^{sym}_E(\mathcal M(t), \mathcal M(\tau))\le  C\(\int_0^\infty e^{-\alpha s}|\gamma(t-s)-\gamma (\tau-s)| \,ds\)^\kappa ,
\end{equation}
where the positive constant $C$ and $0<\kappa<1$ are independent of $t$ and $\tau$ and $\dist^{sym}$ is a symmetric Hausdorff distance between sets in E.
\end{theorem}
\begin{proof}
We follow the approach developed in \cite{EMZ}. As a first step, we note that it is sufficient to construct an exponential attractor not on the whole space $E$, but only on the exponentially attracting ball $\mathcal B^1_R$ of smoother space $E^1$ constructed in Theorem \ref{Th3.main}. Indeed, the already mentioned transitivity of exponential attraction will give then the result for the whole space $E$.
\par
As the next step, we introduce a family of discrete dynamical processes
\begin{equation}\label{4.dc}
\bar U_\tau(m,n):=U_\gamma(nT+\tau,mT+\tau),\  n,m\in\mathbb Z,\ n\ge m
\end{equation}
depending on a parameter $\tau\in\R$. Here $T$ is a sufficiently large positive number which will be fixed later. If we construct exponential attractors
$\mathcal M_{d,\tau} (n)$ for these processes, the desired ”continuous” exponential attractor will be obtained by the standard expression
\begin{equation}\label{4.dcc}
\mathcal M(\tau):=\cup_{s\in[0,T]}U_\gamma(\tau, \tau-s)\mathcal M_{d,\tau-s}(0),
\end{equation}
see \cite{EMZ,MirZel} for details. To verify the existence of the attractors $\mathcal M_{d,\tau} (n)$, we need to check a number of assumptions of the abstract exponential attractor
existence theorem stated in \cite{EMZ}. Namely, we first need to check that
\begin{equation}\label{4.eabs}
\bar U_\tau(n +1,n)\( \mathcal O^{E^1}_\eb(\mathcal B_R^1)\)\subset\mathcal B_R^1,\ \ \tau\in\R
\end{equation}
for sufficiently small $\eb>0$. Here we denote by $\mathcal O^{E^1}_\eb (V )$ the $\eb$-neighbourhood
of the set $V\subset E^1$ in the topology of $E^1$ . Indeed, as follows from the proof of
Theorem \ref{Th3.main}, the dynamical process $U_\gamma(t,\tau)$ is well posed in $E^1$ and possesses a uniformly absorbing ball in it. By this reason,
taking $\eb>0$ small enough
and increasing the radius $R$ if necessary, we get \eqref{4.eabs} for all $\tau\in\R$ if $T> 0$
is large enough.
\par
Second, we need to check the squeezing property (asymptotic smoothing
property for differences of solutions) for the operators $\bar U_\tau(n,m)$. Let $u_1(t)$
and $u_2(t)$ be two solutions of problem \eqref{2.wave} starting at $t=\tau$ from $\mathcal O^{E^1}_\eb(\mathcal B^1_R)$.
Then we know from the dissipative estimate in $E^1$ that
\begin{equation}\label{4.bbound}
 \|\Dt u_i(t)\|_{H^1} + \|u_i(t)\|_{H^2}\le C,\  t\ge\tau,\ i=1,2,
\end{equation}
where the constant $C$ is independent of $t$ and $u_i$. Let $v(t):= u_1(t)-u_2(t)$.
Then, this function solves equation \eqref{2.minus}. In particular, estimate \eqref{2.lip}
holds. We now split $v(t)= v_1(t)+ v_2(t)$, where the function $v_1(t)$ solves
\begin{equation}\label{4.v1}
 \Dt^2 v_1+\gamma(t)\Dt v_1-\Dx v_1+Lv_1=0,\ \ \xi_{v_1}\big|_{t=\tau}=\xi_v\big|_{t=\tau},
\end{equation}
where $L$ is large enough for the assertion of Proposition \ref{Prop3.lin} to be satisfied.
Then, from estimate \eqref{4.lindis}, we have
\begin{equation}\label{4.v1A}
 \|\Dt v_1(t)\|_{L^2}^2+ \|\Nx v_1(t)\|^2_{L^2}\le  Ce^{-\alpha(t-\tau)}\( \|\Dt v(\tau)\|^2_{L^2}+ \|\Nx v(\tau)\|^2_{L^2}\).
\end{equation}
The remainder $v_2(t)$ solves the equation
\begin{equation}\label{4.v2}
\Dt^2v_2+\gamma(t)\Dt v_2-\Dx v_2+Lv_2 = Lv-l(t)v:=h_v(t),\ \ \xi_{v_2}\big|_{t=\tau}=0.
\end{equation}
Using the facts that $u_i(t)$ are bounded in $H^2$
and $f\in C^2$ together with \eqref{2.lip}, one easily get
\begin{equation}
\|h_v(t)\|_{H^1_0}^2\le C\|\Nx v(t)\|^2_{L^2}\le Ce^{K(t-\tau)}\(\|\Dt v(\tau)\|^2_{L^2}+\|\Nx v(\tau)\|^2_{L^2}\).
\end{equation}
Therefore, applying the operator $(-\Dx)^{1/2}$ to both sides of equation \eqref{4.v2}
and using estimate \eqref{4.lindis}, we have
\begin{equation}\label{4.v2K}
\|\Dt v(t)\|_{H^1}^2+\|\Nx v(t)\|^2_{H^1}\le Ce^{K(t-\tau)}\(\|\Dt v(\tau)\|^2_{L^2}+\|\Nx v(\tau)\|^2_{L^2}\),
\end{equation}
where the constants $C$ and $K$ are independent of $u_1$, $u_2$, $t$ and $\tau$.
\par
Let now $\xi_v:= \bar U_\tau(n +1,n)\xi_1-\bar U_\tau(n +1,n)\xi_2$. Then, according to the above estimates $\xi_v$ can be presented as a sum $\xi_v = \xi_{v_1}+\xi_{v_2}$ in such a way that
$$
\|\xi_{v_1}\|_E\le Ce^{-\alpha T}\|\xi_1-\xi_2\|_E, \ \ \|\xi_{v_2}\|_{E^1}\le Ce^{KT}\|\xi_1-\xi_2\|_E
$$
and fixing $T>0$ in such a way that $Ce^{-\alpha T}\le\frac12$, we get the desired squeezing property.
\par
Third, we need to estimate the difference between the processes which correspond to different symbols from the hull $\mathcal H(\gamma)$, say, $\gamma$ and $T_{-s}\gamma$. Let
$\xi_{u_1}(t):=U_\gamma(t,\tau)\xi_\tau$ and
$\xi_{u_2}(t):=U_{T_{-s}\gamma}(t,\tau)\xi_\tau$. Then these functions also satisfy
the uniform estimate \eqref{4.bbound}. Let $v(t):=u_1(t)-u_2(t)$ which solves the
equation
$$
\Dt^2 v+\gamma(t)\Dt v-\Dx v+l(t)v=(\gamma(t)-\gamma(t-s))\Dt u_2(t),\ \ \xi_v\big|_{t=\tau}=0
$$
and arguing as in the proof of estimate \eqref{2.lip}, we end up with the desired estimate
\begin{equation}\label{4.dtime}
\|\xi_v(t)\|_E^2\le C\int_\tau^t e^{K(t-l)}|\gamma(l)-\gamma(l-s)|dl.
\end{equation}
Thus, all of the conditions of the abstract theorem on the existence of an exponential attractor are verified and, therefore, the discrete exponential
attractors $\mathcal M_{d,\tau}(n)$ are constructed, see \cite{EMZ} for more details. Finally, in order to pass from discrete to continuous exponential attractors
using \eqref{4.dcc}, we need to establish that the trajectories of the dynamical process $U_\gamma(t,\tau)$ (started from $\mathcal O^{E^1}_\eb(\mathcal B^1_R))$ are uniformly H\"older continuous in
time (with values in $E$). The continuity and even Lipschitz continuity of $u(t)$ in $H^1$ is obvious since we have the control of $\Dt u$ in $H^1$ , but the H\"older continuity of $\Dt u$ in $L^2$ is a
bit more delicate. Indeed, if we know that
$\gamma\in L^1_b(\R)$ only, from equation \eqref{2.wave} and \eqref{4.bbound}, we will get that $\Dt^2 u\in L^1_b(L^2)$ only and this is not sufficient to get the H\"older continuity in time
 for $\Dt u$. To
overcome this problem, we put an extra condition \eqref{4.gsm} which guarantees
that $\Dt^2 u\in L^{1+\eb}_b(L^2)$ and this gives the desired H\"older continuity of $\Dt u$.
Thus, the above mentioned H\"older continuity is verified and the theorem is proved.
\end{proof}
\begin{rem}
The H\"older continuity \eqref{4.hld} is crucial for the consistency with the autonomous case. Indeed, it guarantees that in the autonomous case $\gamma\equiv const$ the non-autonomous exponential attractor coincides with the autonomous one. Moreover, if $\gamma$ is periodic, quasi-periodic, etc., the same will be true for the non-autonomous attractor as well. Actually, exactly in order to
guarantee this H\"older continuity, we need to assume the extra
regularity \eqref{4.gsm}. This H\"older continuity sometimes may look as an essential restriction, but it cannot be relaxed without losing the consistency with the autonomous case. In particular, without this assumption, we may have a ”pathological” dependence on time of the exponential attractor which correspond to the autonomous case $\gamma\equiv const$.
\end{rem}
\begin{rem}
Since the kernel sections are $\mathcal K_\gamma(\tau)$ always subsets of the corresponding exponential attractors:
$$
\mathcal K_\gamma(\tau)\subset\mathcal M(\tau),
$$
we automatically have the finite-dimensionality of kernel sections $\mathcal K_\gamma(\tau)$ for
all $\tau\in\R$.
\end{rem}

\section{Random dissipation rate}\label{s5}
In this section we weaken the assumptions on the dissipation coefficient $\gamma(t)$ in order to be able to consider the case when the dissipation is random. Indeed, although assumption \eqref{2.Negative} is well adapted to the case of periodic or almost periodic dissipation, the uniformity of averaging with respect to $\tau\in\R$ postulated there is too restrictive if we want to consider chaotic or random dissipation rates and should be relaxed. In this section, we assume that there exists a Borel probability measure $\mu$ on the hull $\mathcal H(\gamma)$ such that the group of shifts $\{T_s,\, s\in\R\}$ is measure preserving and ergodic:
\begin{equation}\label{5.erg}
T_s\mu=\mu,\ s\in\R,\ \ \text{ $\mu$ is {\it ergodic} with respect to $\{T_s,\,s\in\R\}$}.
\end{equation}
We replace \eqref{2.Negative} by weaker assumption
\begin{equation}\label{5.bir}
\int_{\eta\in\mathcal H(\gamma)}\(\int_0^1\(\frac12\gamma_+(s)-\frac{p+2}{p+4}\gamma_-(s)\)\,ds\)\mu(d\eta)=\bar\beta>0.
\end{equation}
Then, by Birkhoff ergodic theorem, for any $\tau\in\R$,
\begin{equation}\label{5.ergcor}
\lim_{T\to\infty}\frac1{2T}\int_{\tau-T}^{\tau+T}\beta(s)\,ds=
\lim_{T\to\infty}\frac1{T}\int_{\tau-T}^{\tau}\beta(s)\,ds=\bar\beta,
\end{equation}
for almost all $\eta\in\mathcal H(\gamma)$,
where $\beta(t)=\beta_\eta(t):=\frac12\eta_+(s)-\frac{p+2}{p+4}\eta_-(s)$. However, in contrast to the previous sections, this limit is {\it not uniform} with respect to $\tau\in\R$ and this leads to essential changes in the theory. In particular, as we will see below, assumption \eqref{5.ergcor} does not guarantee the dissipativity of equation \eqref{4.wave}, moreover, most part of trajectories may be unbounded as $t\to\infty$. To overcome this difficulty, we will use the {\it pullback} attraction property and the theory of pullback/random attractors theory, developed in \cite{CF,CF1,Kloeden},  see also references therein. We start with necessary definition and straightforward results.

\begin{Def} A function $t\to\varphi(t)\in\R$, $t\in\R$ is called {\it tempered} if
\begin{equation}\label{5.tempf}
\lim_{t\to-\infty}e^{\theta t}|\varphi(t)|=0,\ \ \forall\theta>0.
\end{equation}
Analogously, a family of bounded sets $B(t)\subset E$, $t\in\R$ is called tempered if the function $\varphi_B(t):=\|B(t)\|_{E}$ is tempered.
\end{Def}
 We now state the tempered analogues of absorbing and attracting sets as well as the tempered pullback attractors.
 \begin{Def} A family $\Bbb B(t)$, $t\in\R$, of bounded sets in $E$ is called tempered (pullback) absorbing set for the process $U_\eta(t,\tau):E\to E$ if it is tempered and
 for every other tempered family $B(t)$, $t\in\R$, of bounded sets and every fixed $t\in\R$,
 $$
 U_\eta(t,t-s)B(t-s))\subset \Bbb B(t)
 $$
 if $s\ge S(B,t)$ is large enough.
 \par
 A family $\mathcal B(t)$, $t\in\R$, of bounded sets in $E$ is called tempered attracting set if
it  is tempered and for every other tempered family $B(t)$ of bounded sets and every fixed $t\in\R$,
 $$
 \lim_{s\to\infty}\dist_H(U_\eta(t,t-s)B(t-s),\mathcal B(t))=0.
 $$
 \end{Def}
 We are now ready to define the tempered analogue of kernel sections (=pullback attractor).
 \begin{Def} A tempered family of bounded sets $\mathcal K_\eta(t)$ is called tempered pullback attractor if
 \par
 1) $\mathcal K_\eta(t)$ is compact in $E$ for every fixed $t\in\R$;
 \par
 2) It is strictly invariant: $U_\eta(t,\tau)\mathcal K_\eta(\tau)=\mathcal K_\eta(t)$;
\par
 3) It is a tempered (pullback) attracting set.
 \end{Def}
The next proposition is the analogue of Proposition \ref{Prop4.abs} for the tempered case, see \cite{Kloeden} for details.
\begin{prop}\label{Prop5.abs} Let $U_\eta(t,\tau)$ be a dynamical process in $E$ which is continuous in $E$ and possesses a compact (for every $t\in\R$) tempered attracting set. Then $U_\eta(t,\tau)$ possesses a tempered (pullback) attractor $\mathcal K_\eta(t)$ which possesses the following description:
\begin{equation}\label{5.strran}
\mathcal K_\eta(\tau)=\mathcal K_\eta\big|_{t=\tau},
\end{equation}
where $\mathcal K_\eta$ is the set of all complete tempered trajectories of $U_\eta(t,\tau)$ (=the tempered kernel of $U_\eta(t,\tau)$).
\end{prop}
\begin{rem} The assumption that the attracting set is tempered can be relaxed if we assume the existence of a {\it tempered} absorbing set. Then, as usual only asymptotic compactness is necessary to get the existence of an attractor. However, the assumption that the absorbing set is tempered is important in order to have the representation formula \eqref{5.strran}. This property may be lost if the absorbing ball is not tempered and this, in turn, leads to many pathological effects (non-uniqueness of the attractor, etc.).
\end{rem}
The next theorem can be considered as the main result of this section.

\begin{theorem}\label{Th5.main} Let the nonlinearity $f\in C^1(\R)$ satisfy assumptions \eqref{2.f} with $p<4$ and let $\gamma\in L^1_{tr-c}(\R)$ be such that \eqref{5.bir} hold.
Then, for almost all $\eta\in\mathcal H(\gamma)$, the dynamical process $U_\eta(t,\tau)$ associated with equation \eqref{4.wave} possesses a pullback attractor $\mathcal K_\eta(t)$ which is  tempered in $E^\beta$ for some $\beta>0$.
\end{theorem}
\begin{proof} To verify the conditions of Proposition \ref{Prop5.abs} we need to adapt the proofs of Theorems \ref{Th2.main} and \ref{Th3.main} to the random case. This adaptation is almost straightforward, so we briefly indicate below the main difference related with verification that the obtained absorbing/attracting sets are tempered. We start with Theorem \ref{Th2.main} which gives us the existence of an absorbing set.
\par
Indeed, arguing exactly as in the proof of Theorem \ref{Th2.main}, we derive estimate \eqref{2.Gron} with $\beta_\eb(t)=\beta_{\eb,\eta}$ defined by \eqref{2.beta} (where $\gamma$ is replaced by $\eta\in\mathcal H(\gamma)$). Moreover, defining $\bar\gamma$ (resp. $\bar\eta$) using \eqref{2.mol}, we see that the function $\eta\to\int^1_0\beta_{\eb,\eta}(s)\,ds$ is continuous on a compact set $\mathcal H(\gamma)$. By this reason its mean value exists
$$
\int_{\eta\in\mathcal H(\gamma)}\(\int_0^1\beta_{\eb,\eta}(s)\,ds\)\mu(d\eta)=\bar\beta_\eb<\infty
$$
and fixing $\eb>0$ and $\kappa>0$, we may assume that $\bar\beta_\eb>0$ (due to assumption \eqref{5.bir}). The Birkhoff ergodic theorem now gives that, for every $\tau\in\R$
\begin{equation}\label{5.bir1}
\lim_{T\to\infty}\frac1T\int_{\tau-T}^\tau\beta_{\eb,\eta}(s)\,ds=
\lim_{T\to\infty}\frac1T\int_{\tau}^{\tau+T}\beta_{\eb,\eta}(s)\,ds=\bar\beta_\eb>0.
\end{equation}
for almost all $\eta\in\mathcal H(\gamma)$. Of course, $T_s\eta$ also satisfies \eqref{5.bir1} if $\eta$ does, so there is a full measure set $\mathcal H_{erg}\subset \mathcal H(\gamma)$ invariant with respect to $T_s$ such that \eqref{5.bir1} is satisfied for all $\eta\in\mathcal H_{erg}$.
\par
The following technical lemma is crucial for what follows.
\begin{lemma}\label{Lem5.cru} Let $\phi\in L^1_{loc}(\R)$ be such that the function $t\to\int_t^{t+1}|\phi(s)|\,ds$ be tempered and let $\beta_\eb\in L^1_b(\R)$ satisfy \eqref{5.bir1}. Then the function
\begin{equation}\label{5.R}
R(t):=\int_{-\infty}^t\phi(s)e^{-\int_s^t\beta_\eb(l)\,dl}\,ds
\end{equation}
is well-defined and tempered. Moreover, if $A(t)$ is tempered then there exists $\alpha>0$ such that
\begin{equation}\label{5.A}
\lim_{s\to\infty}e^{-\alpha (t-s)}A(t-s)e^{-\int_s^t\beta_\eb(l)\,dl}=0
\end{equation}
for all $t\in\R$.
\end{lemma}
\begin{proof}[Proof of the lemma] Without loss of generality we may assume that $t\le0$. Then we split
$$
\int_s^t\beta_\eb(l)\,dl=\int_s^0\beta_\eb(l)\,dl-\int_t^0\beta_\eb(l)\,dl
$$
and write
$$
|R(t)|\le e^{\int_t^0\beta_\eb(l)\,dl}\int_{-\infty}^t|\phi(s)|e^{-\int^0_s\beta_\eb(l)\,dl}\,ds.
$$
Moreover, due to \eqref{5.bir1}, for every $\nu>0$, we may write
\begin{equation}\label{5.nice}
-(\bar\beta_\eb-\nu)s-C_\nu\le\int_s^0\beta_\eb(l)\,dl\le -(\bar\beta_\eb+\nu)s+C_\nu
\end{equation}
(for some positive $C_\nu$) and the previous estimate reads
$$
|R(t)|\le C'_\nu e^{-(\bar\beta_\eb+\nu)t}\int_{-\infty}^t|\phi(s)|e^{(\bar\beta_\eb-\nu)s}\,ds.
$$
Finally, we utilize the fact that $\int_t^{t+1}|\phi(s)|\,ds$ is tempered. By this reason
$$
\int_{t}^{t+1}|\phi(s)|\,ds\le C_\nu e^{-t\nu/2}
$$
and
$$
|R(t)|\le C_\nu'e^{-(\bar\beta_\eb+\nu)t}\int_{-\infty}^te^{ (\bar\beta_\eb-\frac32\nu)s}\,ds\le
C_\nu''e^{-\frac52\nu t}.
$$
Since $\nu>0$ is arbitrary, the function $R(t)$ is indeed tempered.
\par
To verify the second statement, we note that it is enough to check \eqref{5.A} for $t=0$ only (the uniformity with respect to $t\in\R$ is not assumed in \eqref{5.A}). Then, estimate \eqref{5.nice} immediately gives us \eqref{5.A} for any $\alpha<\bar\beta_\eb$. Thus, the lemma is proved.
\end{proof}
We are now ready to finish the proof of the theorem. Indeed, let $\eta\in\mathcal H_{erg}$ and
\begin{equation}\label{5.rad}
R_\eta(t):=2C\int_{-\infty}^t(1+\|g\|^2_{L^2}+|\eta(s)|)e^{-\int_s^t\beta_{\eb,\eta}(l)\,dl}\,ds,
\end{equation}
where all of the constants and functions are the same as in \eqref{2.Gron}. Then, according to estimate \eqref{2.Gron} and Lemma \ref{Lem5.cru}, the set
$$
\mathcal B^0_E(R_\eta(t)):=\{\xi\in E,\ \mathcal E(\xi)\le R_\eta(t)\}
$$
is a tempered absorbing set for the process $U_\eta(t,\tau)$ associated with wave equation \eqref{4.wave}. Moreover, the fact that $R_{\eta}(t)$ solves the integral equation
$$
R_\eta(t)=R_\eta(\tau)e^{-\int_\tau^t\beta_{\eb,\eta}(s)\,ds}+2C\int_{\tau}^t(1+\|g\|^2_{L^2}+|\eta(s)|)e^{-\int_s^t\beta_{\eb,\eta}(l)\,dl}\,ds
$$
for all $\tau\le t$ (compare with \eqref{2.Gron}), gives that this absorbing set is invariant:
$$
U_{\eta}(t,\tau)\mathcal B^0_E(R_{\eta}(\tau)\subset\mathcal B^0_E(R_{\eta}(t)).
$$
To obtain the tempered compact attracting set we need to get random analogues of estimates proved in Section \ref{s3}. We first observe that the analogue of estimate \eqref{4.lindis} for the random case reads
$$
\mathcal E_v(t)\le C\mathcal E_v(\tau)e^{-\int_\tau^t\beta_{\eb,\eta}(l)\,dl}+C\(\int_s^te^{-\beta_{\eb,\eta}(l)\,dl}\|h(s)\|_{L^2}\,ds\)^2
$$
and big constant $L$ can be chosen uniformly with respect to $\eta\in\mathcal H(\gamma)$, see \eqref{3.grr} and \eqref{3.grrr}. Thus, the solution $\xi_v(t)$ of \eqref{3.lind} starting from a tempered absorbing set $\xi_u(\tau)\in\mathcal B^0_E(R_\eta(\tau))$ will be exponentially decaying as $\tau\to-\infty$ according to Lemma \ref{Lem5.cru}.
\par
By this reason, to verify the existence of a tempered attracting set, it is sufficient to check that the function $h_u(t)$ in the right-hand side of \eqref{3.rem} is tempered as $\tau\to-\infty$. To be more precise, we need to show that the function $\tau\to \|h_u\|_{L^1(\tau,\tau+1;H^\beta)}$ is tempered as $\tau\to-\infty$. But this is an immediate corollary of the fact that $\xi_u(t)$ belongs to a tempered absorbing ball and estimate \eqref{2.polc}, see the proof of Theorem \ref{Th3.main}.
\par
Thus, the tempered compact attracting set for $U_\eta(t,\tau)$ is constructed (as a tempered ball in smoother space $E^\beta$, $\beta>0$) and the theorem is proved.
\end{proof}
\begin{rem} Arguing analogously to the proof of Theorem \ref{Th3.main}, we may prove that the attractor $\mathcal K_\eta(t)$ is a tempered set in $E^1$, but to this end we need the tempered version for transitivity of exponential attraction. Since we need not this regularity for what follows, we prefer not to discuss this topic here.
\par
We also note that the radius $R_\eta(t)$ of the absorbing ball is {\it measurable} (for every fixed $t$ as a function of $\eta\in\mathcal H(\gamma)$. Indeed, it can be presented as an almost everywhere limit of continuous functions:
$$
R_\eta(t):=\lim_{n\to\infty}\int_{-n}^t(1+\|g\|^2_{L^2}+|\eta(s)|)e^{-\int_s^t\beta_{\eb,\eta}(l)\,dl}\,ds.
$$
By this reason, the set-valued function $\eta\to \mathcal B_E^0(R_\eta(t))$ is measurable. Then, by standard arguments, see e.g., \cite{CF,ShirZ}, the attractor $\mathcal K_\eta(t)$ is also measurable as a set-valued function $\eta\to\mathcal K_\eta(t)$ for every fixed $t$.
\end{rem}
We are now ready to complete the construction of a random attractor for equation \eqref{4.wave}. We first remind its definition adapted to our case, see \cite{Kloeden} for more details. First, by definition, a random tempered set is a measured set-valued function $\eta\to B(\eta)\subset E$ such that the function $t\to \|B(T_t\eta)\|_E$ is tempered for almost all $\eta$.
\begin{Def}\label{Def5.ran} A random tempered set $\eta\to\mathcal A(\eta)\subset E$ is a random attractor for the family of processes $U_\eta(t,\tau)$, $\eta\in\mathcal H(\gamma)$ if
\par
1) $\mathcal A(\eta)$ is compact for almost all $\eta\in\mathcal H(\gamma)$;
\par
2) It is strictly invariant: $U_{\eta}(t,0)\mathcal A(\eta)=A(T_t\eta)$ for all $t\ge0$ and almost all~$\eta$.
\par
3) For any tempered random set $\eta\to B(\eta)$ and almost all $\eta\in\mathcal H(\gamma)$,
\begin{equation}\label{5.vzad}
\lim_{\tau\to-\infty}\dist_E(U_\eta(0,\tau)B(T_\tau\eta),\mathcal A(\eta))=0.
\end{equation}
\end{Def}
\begin{cor}\label{Cor5.main}
  Let the assumptions of Theorem \ref{Th5.main} hold. Then the family of processes $U_\eta(t,\tau):E\to E$, $\eta\in\mathcal H(\gamma)$ possesses a random attractor $\mathcal A(\eta)$ which is a tempered random set in $E^\beta$ for some $\beta>0$.
\end{cor}
\begin{proof}
Indeed, we may define
\begin{equation}
\mathcal A(\eta):=\begin{cases} \mathcal K_\eta(0),\ \ \eta\in\mathcal H_{erg}(\gamma),\\
                  \varnothing,\ \ \eta\notin\mathcal H_{erg}(\gamma),\end{cases}
\end{equation}
where $\mathcal K_\eta(t)$ is the tempered pullback attractor constructed in Theorem \ref{Th5.main}. Then, all assertions of Definition \ref{Def5.ran} are verified above and the desired attractor is constructed.
\end{proof}
\begin{rem}\label{Rem5.vperd} As we have already mentioned, a pullback attractor in general fails to attract bounded sets forward in time. The situation is much better for the case of {\it random} attractors where the forward convergence in measure usually holds, see \cite{CF}. Indeed, the Lebesgue dominated convergence theorem allows us to conclude from almost everywhere convergence \eqref{5.vzad} that
$$
\int_{\mathcal H(\gamma)}\frac{\dist_E(U_\eta(0,\tau)B(T_\tau\eta),\mathcal A(\eta))}{1+\dist_E(U_\eta(0,\tau)B(T_\tau\eta),\mathcal A(\eta))}\mu(d\eta)\to0
$$
as $\tau\to-\infty$. Using now the translation identity and the fact that $T_\tau$ is measure preserving, after the change of variable $\eta\to T_\tau\eta$, we arrive at
$$
\int_{\mathcal H(\gamma)}
\frac{\dist_E(U_\eta(\tau,0)B(\eta),\mathcal A(T_\tau\eta))}
{1+\dist_E(U_\eta(\tau,0)B(\eta),\mathcal A(T_\tau\eta))}\mu(d\eta)\to0
$$
as $\tau\to+\infty$. It remains to note that the last convergence is equivalent to the desired forward convergence in measure
\begin{equation}\label{5.vperd}
\mu\!-\!\!\lim_{\tau\to\infty} \dist_E(U_\eta(0,\tau)B(\eta),\mathcal A(T_\tau\eta))=0.
\end{equation}
\end{rem}
We complete this section by considering a natural model example when the dynamics $T_s:\mathcal H(\gamma)\to\mathcal H(\gamma)$ is determined by the Bernoulli shift dynamics. Namely, let $\Gamma:=\{a,-b\}^{\Bbb Z}$ be two sided Bernoulli scheme with two symbols $\{a,-b\}$ and let
$$
(T_l\gamma)(n)=\gamma(n+l),\ \ l\in\Bbb Z,\   \gamma=(\cdots,\gamma_{-n},\cdots,\gamma_n,\cdots)\in\Gamma
$$
be the associated Bernoulli process. We endow the set $\Gamma$ by the Tichonoff topology and by the Borel probability product measure $\mu$ generated by the probability measure on a cross section
$$
\mu(\{a\})=q,\ \ \mu(\{-b\})=1-q,\ \ q\in(0,1).
$$
Then, as known, see e.g. \cite{KH}, the dynamical system $(T_l,\Gamma,\mu)$ is transitive and ergodic. We extend any
element $\gamma\in\Gamma$ to a function $\mathcal R(\gamma)\in L^\infty(\R)$ as follows:
\begin{equation}
\mathcal R(\gamma)(t):=\gamma_{[t]}
\end{equation}
where $[t]$ is an integer part of $t$. Then, obviously,
$$
T_l\mathcal R(\gamma)=\mathcal R(T_l\gamma)
$$
for all $l\in\Bbb Z$. It is not difficult to show that the Tychonoff topology on $\Gamma$ induces the $L^1_{loc}$-topology on $\mathcal R(\gamma)$, so $\mathcal R(\gamma)$ is translation compact for any $\gamma\in\Gamma$. Moreover, if we take any transitive trajectory $\gamma\in\Gamma$, the hull $\mathcal H(\mathcal R(\gamma))$ will contain the image of the whole Bernoulli scheme $\Gamma$.
\par
Thus, the elements of the hull $\mathcal H(\gamma):=\mathcal H(\mathcal R(\gamma))$ are parameterized by the elements $\eta\in\Gamma$ and the dynamics on the hull is equivalent to the classical Bernoulli shift dynamics. Being pedantic, we first need to extend the discrete Bernoulli shifts to the continuous ones acting on the extended space $\Gamma\times[0,1]$ (to parameterize the non-integer shifts) and only after that establish the equivalence, but to avoid the technicalities we omit this step and will identify (with a slight abuse of rigor) the element $\gamma\in\Gamma$ with $\mathcal R(\gamma)\in L^\infty(\R)$ as well as discrete shifts on $\Gamma$ with continuous shifts on $\mathcal H(\gamma)$.
\par
The key condition \eqref{5.bir} now reads
\begin{equation}\label{5.beer}
\int_{\eta\in\mathcal H(\gamma)}\int_0^1\(\frac12\gamma_+(s)-\frac{p+2}{p+4}\gamma_-(s)\)ds\,\mu(d\eta)=\frac12 aq-\frac{p+2}{p+4}b(1-q)>0.
\end{equation}
Thus, we have finally proved the following result.

\begin{theorem}\label{Th5.ber} Let the exponents $a$ and $b$ and the probability $q$ satisfy assumption \eqref{5.beer}. Then, the wave equation \eqref{4.wave} with damping rate $\eta$ generated by the above described Bernoulli process possesses a random attractor $\mathcal A(\eta)$ in the energy space $E$.
\end{theorem}

\section{Infinite-dimensionality of random attractors: a toy example}\label{s6}

The results on the existence of random attractors stated in Section \ref{s5} look more or less standard, see e.g., \cite{Kloeden} and references therein. Nevertheless, there is an {\it essential} difference between this case and the random attractors considered in the above mentioned works. This difference becomes transparent if we try to compute the mean of the size of tempered absorbing set constructed in Theorem \ref{Th5.main}. Indeed, taking the mean of the expression \eqref{5.rad}, we get
\begin{multline}\label{7.besk}
\int_{\eta\in\mathcal H(\gamma)}R_\eta(0)\mu(d\eta)\sim C\int_{\eta}\(\sum_{k=-\infty}
^0e^{-\eb k-2\sum_{l=-k}^0\frac12\eta_+(l)-\frac{p+2}{p+4}\eta_-(l)}\)\mu(d\eta)=\\=
\sum_{k=-\infty}^0e^{\eb (-k+1)}\int_\eta\(\Pi_{l=-k}^0e^{-\eta_+(l)+\frac{2(p+2)}{p+4}\eta_-(l)}\)\mu(d\eta)=\\=\sum_{k=-\infty}^0e^{-\eb k}\(\int_\eta\(e^{-\eta_+(0)+\frac{2(p+2)}{p+4}\eta_-(0)}\)\mu(d\eta)\)^{-k+1}=\\=\sum_{k=0}^\infty e^{(k+1)(\eb+\ln(e^{-a}q+e^{\frac{2(p+2)}{p+4}b}(1-q)))},
\end{multline}
where we have used that $\eta(n)$ and $\eta(m)$ are independent as random variables if $m\ne n$. Thus, the expression in the right-hand side will be finite if and only if
$$
\ln(e^{-a}q+e^{\frac{2(p+2)}{p+4}b}(1-q))<-\eb<0.
$$
So, if the probability $q$ and the exponents $a>0$ and $b>0$ be such that
\begin{equation}\label{6.inf}
\ln(e^{-a}q+e^{\frac{2(p+2)}{p+4}b}(1-q))>0>-aq+\frac{2(p+2)}{p+4}b(1-q),
\end{equation}
the obtained energy bound $R_\eta(t)$ will have {\it infinite} mean (in contrast to \cite{CF,CF1} where this mean is always finite).
\par
The absence of the first momentum for the energy may make a drastic impact on the underlying limit dynamics, in particular, making it {\it infinite-dimensional}. This observation is based on the following standard counterpart of the Birkhoff ergodic theorem, see e.g., \cite{Aaronson} for more delicate results in this direction.
\begin{prop}\label{Prop6.cerg} Let $(X,\mu)$ be a compact metric space with Borel probability measure $\mu$ on it and let $T:X\to X$ be an  ergodic map. Assume that $f:X\to\R$ be a non-negative measurable function such that $\int_X f(x)\,\mu(dx)=\infty$. Then
\begin{equation}
\lim_{N\to\infty}\frac1N\sum_{n=1}^Nf(T^nx)=\infty
\end{equation}
for almost all $x\in X$.
\end{prop}
\begin{proof} Indeed, by the definition of Lebesgue integration, there exist a sequence of simple functions $f_l\in L^1(X,\mu)$ such that
$$
0\le f_l(x)\le f(x),\ \ f_l(x)\to f(x) \text{ almost everywhere }
$$
and $\int_Xf_l(x)\,\mu(dx)\to\infty$. Take $x\in X$ such that Birkhoff ergodic theorem holds for it for all $l$. Then
$$
\lim_{N\to\infty}\frac1N\sum_{n=1}^Nf(T^nx)\ge\lim_{N\to\infty}\frac1N\sum_{n=1}^Nf_l(T^nx)=
\int_Xf_l(x)\mu(dx)\to\infty
$$
as $l\to\infty$ and the proposition is proved.
\end{proof}
The proved statement allows us to expect that, in the case where the momentums of the energy do not exist, the global Lyapunov exponents and Lyapunov dimension of the attractors also may be infinite (since the ergodic sums involving energy are usually present at least in the estimates for these exponents, see \cite{CF1,Deb} for the details). This in turn allows us to expect that the Hausdorff and fractal dimension of the attractor may be also infinite. To be more precise, we state the following conjecture.
\begin{conjecture}\label{Con6.inf} Let the exponents $p$, $a$, $b$ and probability $q$ satisfy \eqref{6.inf}. Then, there are nonlinearity $f$ and right-hand side $g$ satisfying the assumptions of Theorem \ref{Th5.ber} such that the associated random attractor $\mathcal A(\eta)$ has an infinite Hausdorff and fractal dimensions in $E$.
\end{conjecture}
Up to the moment, we are unable to prove or disprove this conjecture. However, to support this conjecture, we conclude the section by a simplified model example where this conjecture holds.
\begin{example}\label{Ex6.inf} Let $H=l_2$ (space of square summable sequences) and consider the random dynamical system in $H$ generated by the following equations:
\begin{equation}\label{6.odes}
\frac d{dt} u_1+\eta(t)u_1(t)=1,\ \ \frac d{dt}u_k+k^4u_k=u_1(t)u_k-u_k^3, \ k=2,3,\cdots,
\end{equation}
where $u=(u_1,u_2,\cdots)\in H$ and $\eta\in\Gamma$ is exactly the Bernoulli process used in Theorem \ref{Th5.ber}. The first equation of this system models the energy evolution of \eqref{4.wave} and the rest equations give some coupling of the first equation with a parabolic PDE.
\par
System of ODEs \eqref{6.odes} is simple enough to be solved explicitly. In particular, if $aq-(1-q)b>0$,
\begin{equation}\label{6.good}
u_1(t)=u_{1,\eta}(t)=\int_{-\infty}^te^{-\int_s^t\eta(l)\,dl}\,ds
\end{equation}
is a unique tempered complete solution of the first equation (for almost all $\eta\in\Gamma$). Moreover, arguing as in \eqref{7.besk}, we get
\begin{equation}\label{6.uinf}
\int_{\eta\in\Gamma}u_{1,\eta}(0)\mu(d\eta)=\infty
\end{equation}
if
\begin{equation}\label{6.infinf}
\ln(qe^{-a}+(1-q)e^{b})>0>-aq+(1-q)b.
\end{equation}
We claim that under this assumption the random attractor associated with \eqref{6.odes} is infinite-dimensional.
\begin{theorem}\label{Th6.main} Let the exponents $a,b>0$ and the probability $q\in(0,1)$ satisfy \eqref{6.infinf}. Then the random attractor $\mathcal A(\eta)$ for system \eqref{6.infinf} in $H$ has infinite Hausdorff and fractal dimensions in $H$:
\begin{equation}\label{6.dim}
\dim_H(\mathcal A(\eta),H)=\dim_f(\mathcal A(\eta),H)=\infty
\end{equation}
for almost all $\eta\in\Gamma$.
\end{theorem}
\begin{proof} We first briefly discuss why this random attractor exists. The existence of a random tempered absorbing set for the first component $u_1$ of system \eqref{6.odes} follows from the explicit formula for the solution and Lemma \ref{Lem5.cru}. Let us now assume that $u_1(\tau)$ is already in this absorbing set and get the estimates for $u_k$. To this end, multiplying the $k$th equation by $\operatorname{sgn}(u_k)$ and taking a sum, after the standard estimates, we get
$$
\frac d{dt}\(\sum_{k=2}^\infty|u_k|\)+\sum_{k=2}^\infty k^4|u_k|\le C\(\sum_{k=2}^\infty\frac1{k^2}\)|u_1|\le C|u_1(t)|.
$$
Integrating this inequality and using that $u_1(t)$ is tempered, we get a tempered absorbing ball for $u$ in $l_1\subset H$. To obtain a compact absorbing set, it is enough to use the parabolic smoothing property in a standard way. Thus, the existence of a random attractor $\mathcal A(\eta)$ is verified.
\par
Recall also that a random attractor consists of all complete tempered trajectories of the system considered, so it is enough to find all such trajectories. The first equation is linear and is independent of the other equations, so such trajectory is unique and is given by \eqref{6.good}. Thus, to find $\mathcal K_\eta$, we need to fix $u_1(t)=u_{1,\eta}(t)$ in the other equations of \eqref{6.good} and find the tempered attractor $\mathcal A^k(\eta)$, $k=2,3,\cdots$ for every component of \eqref{6.good} separately. Then, the desired attractor $\mathcal A(\eta)$ for the whole system is presented as a product
\begin{equation}\label{6.pro}
\mathcal A(\eta)=\{u_{1,\eta}(0)\}\times\otimes_{k=2}^\infty\mathcal A^k(\eta).
\end{equation}
Moreover, since the attractor is always connected, every $\mathcal A^k(\eta)$ is a closed interval, so to prove the desired infinite-dimensionality, it is enough to prove that $\mathcal A^k(\eta)\ne\{0\}$ for all $k$. In other words, we need to find a non-zero tempered trajectory $u_k=u_k(t)$ for the equation
\begin{equation}\label{6.k}
\frac d{dt} u_k+k^4u_k=u_{1,\eta}(t)u_k-u_k^3
\end{equation}
using that $u_{1,\eta}(t)$ is tempered and have an infinite mean.
\par
To this end, we first note that every solution of equation \eqref{6.k} is either tempered as $t\to-\infty$ or blows up backward in time (this can be easily shown by comparison using the fact that $u_{1,\eta}(t)$ is tempered), so any complete solution $u_k(t)$, $t\in\R$ is automatically tempered.
\par
We construct this solution by solving equation \eqref{6.k} explicitly. Namely,
$$
u_{k,\eta}(t):=\frac1{\sqrt{2\int_{-\infty}^te^{2\int_s^t(k^4-u_{1,\eta}(l))\,dl}\,ds}}.
$$
Indeed, the finiteness of the integral is guaranteed by \eqref{6.uinf} and Proposition \eqref{Prop6.cerg} (analogously to the proof of Lemma \ref{Lem5.cru}). Thus, we have proved that $\mathcal A^k(\eta)\ne \{0\}$ for almost all $k$ and the theorem is proved.
\end{proof}
\begin{rem} Actually, we have found explicitly the random attractor for the previous example:
$$
\mathcal A(\eta)=\{u_{1,\eta}(0)\}\times\otimes_{k=2}^\infty[-u_{k,\eta}(0),u_{k,\eta}(0)].
$$
It would be interesting to compute (e.g., using this expression) the typical Kolmogorov's entropy of this infinite-dimensional attractor.
\end{rem}
\end{example}

\end{document}